\documentclass{amsart}
\usepackage[margin=2cm]{geometry}

\usepackage[parfill]{parskip}    % Activate to begin paragraphs with an empty line rather than an indent
\usepackage{graphicx, txfonts}
\usepackage{epstopdf}
\usepackage[pagebackref]{hyperref} % for pdflatex, this command tells the bibliography to point back to the page where each thing is cited.
\usepackage{amsthm, amssymb, amsfonts, amsmath, enumerate, xy, mathrsfs}
\input xy
\xyoption{all}
\newtheorem{theorem}{Theorem}[section]

\newtheorem{proposition}[theorem]{Proposition}
\newtheorem{prop}[theorem]{Proposition}
\newtheorem{lemma}[theorem]{Lemma}
\newtheorem{corollary}[theorem]{Corollary}
\newtheorem{main}{Theorem}

\theoremstyle{definition}

\newtheorem{defn}[theorem]{Definition}

\newtheorem{example}[theorem]{Example}

\newtheorem{convention}[theorem]{Convention}

\theoremstyle{remark}
\newtheorem{remark}[theorem]{Remark}

\newcommand{\N}{\mathbb{N}}

\newcommand{\R}{\mathbb{R}}

\newcommand{\F}{\mathbb{F}}

\newcommand{\M}{\mathcal{M}}
\newcommand{\C}{\mathcal{C}}

\newcommand{\D}{\mathcal{D}}

\newcommand{\sF}{\mathscr{F}}

\newcommand{\sK}{\mathscr{K}}
\newcommand{\sN}{\mathscr{N}}

\newcommand{\sP}{\mathscr{P}}

\renewcommand{\emptyset}{\varnothing}

\renewcommand{\tilde}[1]{\widetilde{#1}}
\renewcommand{\hat}[1]{\widehat{#1}}

\DeclareMathOperator{\dom}{dom}

\DeclareMathOperator{\colim}{colim}

\DeclareMathOperator{\Arr}{Arr}
\DeclareMathOperator{\Sym}{Sym}
\DeclareMathOperator{\CMon}{CMon}

\DeclareMathOperator{\Ho}{Ho}

\DeclareMathOperator{\ho}{Ho}

\newcommand{\po}{\ar@{}[dr]|(.7){\Searrow}}
\newcommand{\pb}{\ar@{}[dr]|(.3){\Nwarrow}}
\DeclareMathOperator{\map}{map}

\newcommand{\cat}[1]{\mathcal{#1}}

\newcommand{\boxprod}{\mathbin\square}
\newcommand{\Sp}{\mathscr{S}}

\begin{document}
\title{Monoidal Bousfield Localizations and Algebras over Operads}

\author{David White}
\address{Department of Mathematics and Computer Science \\ Denison University
\\ Granville, OH 43023}
\email{david.white@denison.edu}

\begin{abstract}
We give conditions on a monoidal model category $\mathcal{M}$ and on a set of maps $\mathcal{C}$ so that the Bousfield localization of $\mathcal{M}$ with respect to $\mathcal{C}$ preserves the structure of algebras over various operads. This problem was motivated by an example that demonstrates that, for the model category of equivariant spectra, preservation does not come for free, even for cofibrant operads. We discuss this example in detail and provide a general theorem regarding when localization preserves $P$-algebra structure for an arbitrary operad $P$.

We characterize the localizations that respect monoidal structure and prove that all such localizations preserve algebras over cofibrant operads. As a special case we recover numerous classical theorems about preservation of algebraic structure under localization, in the context of spaces, spectra, chain complexes, and equivariant spectra. We also provide several new results in these settings, and we sharpen a recent result of Hill and Hopkins regarding preservation for equivariant spectra. To demonstrate our preservation result for non-cofibrant operads, we work out when localization preserves commutative monoids and the commutative monoid axiom, and again numerous examples are provided. Finally, we provide conditions so that localization preserves the monoid axiom.
\end{abstract}

\maketitle

\section{Introduction}

Bousfield localization is a powerful tool in homotopical algebra, with classical applications to homology localization for spaces and spectra \cite{bousfield-localization-spaces-wrt-homology}, to cellularization and nullification \cite{farjoun}, and to $p$-localization and completion \cite{bous79}. Hirschhorn generalized the machinery of Bousfield localization to the setting of model categories \cite{hirschhorn}, inverting any class of morphisms generated by a set. This general framework has seen tremendous application: it allows for the passage from levelwise model structures to stable model structures \cite{hovey-shipley-smith}, it is used to set up motivic homotopy theory \cite{hovey-spectra}, and it allows for the study of combinatorial model categories via simplicial presheaves \cite{htt} (Appendix A.2). The interplay between left Bousfield localization and monoidal structure has often proven fruitful, e.g. to put an $E_\infty$-algebra structure on connective $K$-theory \cite{EKMM}, for homotopy theoretic computations involving generalized Eilenberg-Maclane spaces \cite{bous-spaces, eternal-preprint, farjoun}, and, recently, to create an equivariant spectrum with a certain periodicity that is used to resolve the Kervaire invariant one problem \cite{kervaire}. In this paper, we further the study of the interplay between left Bousfield localization and monoidal model categories, we provide conditions so that left Bousfield localization preserves algebras over operads (and several important model categorical axioms), and we apply our results to numerous classical and new examples of interest.

Structured ring spectra have had numerous applications in stable homotopy theory \cite{EKMM, hovey-shipley-smith, MMSS}. Nowadays, structured ring spectra are often thought of as algebras over operads acting in any of the monoidal model categories for spectra. It is therefore natural to ask the extent to which Bousfield localization preserves such algebraic structure. For Bousfield localizations at homology isomorphisms this question is answered in \cite{EKMM} and \cite{MMSS}. The case for spaces is subtle and is addressed in \cite{bous-spaces}, \cite{eternal-preprint}, and \cite{farjoun}. More general Bousfield localizations are considered in \cite{CGMV}.

The preservation question may also be asked in the context of equivariant and motivic spectra, and it turns out the answer is far more subtle. In Example \ref{example:hill}, we discuss an example of a naturally occurring Bousfield localization of equivariant spectra that preserves the type of algebraic structure considered in \cite{EKMM} but fails to preserve the equivariant commutativity needed for the landmark results in \cite{kervaire}. We generalize this example in \ref{example:hill-ober-version}.

In order to understand this and related examples, we find conditions on a model category $\M$ and on a class of maps $\C$ so that the left Bousfield localization $L_\C$ with respect to $\C$ preserves the structure of algebras over various operads. After a review of the pertinent terminology in Section \ref{sec:background} we give our general preservation result in Section \ref{sec:main-preservation}, which we state here for the reader's convenience.

\begin{main} \label{main:preservation}
Let $\M$ be a monoidal model category and $\cat C$ a class of morphisms such that the Bousfield localization $L_\C(\M)$ exists and is a monoidal model category. Let $P$ be an operad valued in $\M$. If the subcategories of $P$-algebras in $\M$ and in $L_\C(\M)$ inherit transferred semi-model structures from $\M$ and $L_\C(\M)$ (with weak equivalences and fibrations defined via the forgetful functor) then $L_\C$ preserves $P$-algebras.
\end{main} 

In Section \ref{sec:monoidal-bous-loc} we characterize when $L_\C(\M)$ is a monoidal model category, proving the following theorem. 

\begin{main}
Suppose $\M$ is a cofibrantly generated monoidal model category in which cofibrant objects are flat. Then $L_\C$ is a monoidal Bousfield localization if and only if every map of the form $f \otimes id_K$, where $f$ is in $\C$ and $K$ is cofibrant, is a $\C$-local equivalence. If the domains of the generating cofibrations $I$ are cofibrant, it suffices to consider $K$ in the set of domains and codomains of the morphisms in $I$.
\end{main}

In Section \ref{sec:preservation-over-cofibrant}, we apply Theorem \ref{main:preservation} to numerous model categories of interest, obtaining preservation results for $\Sigma$-cofibrant operads such as $A_\infty$ and $E_\infty$ in model categories of spaces, spectra, chain complexes, and equivariant spectra. We recover several classical preservation results, and prove several new preservation results. A lattice of equivariant operads is presented, that interpolate between non-equivariant $E_\infty$-algebra structure and equivariant $E_\infty$-algebra structure. We apply our results to determine which localizations preserve the type of algebraic structure encoded by these operads. This new collection of operads is different from the $N_\infty$-operads of \cite{blumberg-hill} (that interpolate between equivariant $E_\infty$-algebra structure and genuine commutative structure), and a common generalization of both types of operads is discussed in \cite{gutierrez-white-equivariant}. Example \ref{example:hill} demonstrates that it is possible to preserve equivariant $E_\infty$-algebra structure, but fail to preserve genuine commutative structure. This motivates the latter half of the paper. 

In Section \ref{sec:preservation-of-commutative}, we turn to preservation of structure over non-cofibrant operads, specifically, preservation of commutative monoids. For categories of spectra the phenomenon known as rectification means that preservation of strict commutativity is equivalent to preservation of $E_\infty$-structure, but for general model categories (including equivariant spectra) there can be Bousfield localizations that preserve the latter type of structure and not the former. In the companion paper \cite{white-commutative-monoids} we introduced a condition on a monoidal model category called the \textit{commutative monoid axiom}, that guarantees that the category of commutative monoids inherits a model structure. We build on this work in Section \ref{sec:preservation-of-commutative} by providing conditions on the maps in $\C$ so that Bousfield localization preserves the commutative monoid axiom, proving the following theorem. 

\begin{main} \label{main:comm}
Assume $\M$ is a cofibrantly generated monoidal model category satisfying the strong commutative monoid axiom and with domains of the generating cofibrations cofibrant. Suppose that $L_\C(\M)$ is a monoidal Bousfield localization with generating trivial cofibrations $J_\C$. Then $L_\C(\M)$ satisfies the strong commutative monoid axiom if and only if $\Sym^n(f)$ is a $\C$-local equivalence for all $n \in \N$ and for all $f \in J_\C$. This occurs if and only if $\Sym(-)$ preserves $\C$-local equivalences between cofibrant objects.
\end{main}

In Section \ref{sec:commutative-preservation-applications}, we apply Theorems \ref{main:preservation} and \ref{main:comm} to obtain preservation results for commutative monoids in spaces, spectra, chain complexes, and equivariant spectra. We recover classical preservation results, and several new preservation results, including Theorem \ref{thm:recover-hill-preservation}, which sharpens and generalizes the main theorem of \cite{hill-hopkins}. Finally, in Section \ref{sec:localization-and-monoid-ax} we provide conditions so that $L_\C(\M)$ satisfies the monoid axiom when $\M$ does, proving the following theorem (see Section \ref{sec:localization-and-monoid-ax} for definitions of the unfamiliar terms, from \cite{batanin-berger}). 

\begin{main} \label{main:monoid}
Suppose $\M$ is a cofibrantly generated, left proper, $h$-monoidal model category such that the (co)domains of $I$ are cofibrant and are finite relative to the $h$-cofibrations and cofibrant objects are flat. Then for any monoidal Bousfield localization $L_\C$, the model category $L_\C(\M)$ satisfies the monoid axiom.
\end{main}

Because Theorem \ref{main:preservation} only requires transferred semi-model structures, Theorem \ref{main:monoid} is not required for preservation, but is required in order to have a comprehensive study of the relationship between left Bousfield localization and monoidal structure, as the monoid axiom is often required for purposes other than transferring model structures. As always, we provide applications of Theorem \ref{main:monoid} to the examples of interest in this paper: spectra, spaces, chain complexes, and equivariant spectra. Roughly half of this paper consists of applications to these examples. In the setting of Theorem \ref{main:monoid}, this requires some new results, including a verification that the commonly studied model structures on symmetric spectra are $h$-monoidal, and the introduction of new model structures on equivariant spectra that are combinatorial and $h$-monoidal.

\subsection*{Acknowledgments}

The author would like to gratefully acknowledge the support and guidance of his advisor Mark Hovey as this work was completed. The author is also indebted to Mike Hill, Carles Casacuberta, Justin Noel, and Clemens Berger for numerous helpful conversations. The author thanks Clark Barwick for catching an error in an early version of this work, Martin Franklin for suggesting applications of this work to simplicial sets, and Boris Chorny for suggesting a simplification in the proof of Theorem \ref{thm:monoid-axiom-loc}. This draft was improved by comments from Javier Guti\'{e}rrez, Brooke Shipley, and Cary Malkiewich. This paper also possesses a User's Guide \cite{white-user-guide}, where the interested reader can learn more, created as part of the User's Guide Project presented in \cite{user-guide-project}.

\section{Preliminaries}
\label{sec:background}

We assume the reader is familiar with basic facts about model categories. Excellent introductions to the subject can be found in \cite{dwyer-spalinski}, \cite{hirschhorn}, and \cite{hovey-book}. Throughout the paper we will assume $\M$ is a cofibrantly generated model category, with generating cofibrations $I$ and generating trivial cofibrations $J$.

Let $I$-cell denote the class of transfinite compositions of pushouts of maps in $I$, and let $I$-cof denote retracts of such. In order to run the small object argument, we will assume the domains $K$ of the maps in $I$ (and $J$) are $\kappa$-small relative to $I$-cell (resp. $J$-cell), i.e. given a regular cardinal $\lambda \geq \kappa$ and any $\lambda$-sequence $X_0\to X_1\to \dots$ formed of maps $X_\beta \to X_{\beta+1}$ in $I$-cell, then the map of sets $\colim_{\beta < \lambda} \M(K,X_\beta) \to \M(K,\colim_{\beta < \lambda} X_\beta)$ is a bijection. An object is small if there is some $\kappa$ for which it is $\kappa$-small. See Chapter 10 of \cite{hirschhorn} for a more thorough treatment of this material. For any object $X$ we have a cofibrant replacement $QX \to X$ and a fibrant replacement $X\to RX$.

We will at times also need the hypothesis that $\M$ possesses sets of generating (trivial) cofibrations $I$ and $J$ with domains (hence codomains) cofibrant. This hypothesis is satisfied by all model categories of interest in this paper, but does not come for free, even for combinatorial model categories $\M$. An example, due to Carlos Simpson, is discussed in Remark \ref{remark:simpson-example}. A method for finding sets $I$ and $J$ with cofibrant domains is given in Lemma \ref{lemma:I'UJ}.

Our model category $\M$ will be a monoidal category with product $\otimes$ and unit $S \in \M$. This means $\M$ satisfies the following two axioms.

\begin{enumerate}
  \item Unit Axiom: For any cofibrant $X$, the map $QS\otimes X \to S\otimes X\cong X$ is a weak equivalence.
  \item Pushout Product Axiom: Given any $f:X_0\to X_1$ and $g:Y_0\to Y_1$ cofibrations, $f\boxprod g: X_0\otimes Y_1 \coprod_{X_0\otimes Y_0}X_1\otimes Y_0 \to X_1\otimes Y_1$ is a cofibration. Furthermore, if either $f$ or $g$ is trivial then $f\boxprod g$ is trivial.
\end{enumerate}

Note that the pushout product axiom is equivalent to the statement that $-\otimes-$ is a Quillen bifunctor. Furthermore, it is sufficient to check the pushout product axiom on the generating maps $I$ and $J$, by Proposition 4.2.5 of \cite{hovey-book}. When we need $\M$ to be a simplicial model category, we require the SM7 axiom, which is analogous to the pushout product axiom.
 %Similar axioms define what it means for $\M$ to be a simplicial model category. In this context the pushout product axiom is called the SM7 axiom. It retains the same statement but where one of the maps $f,g$ is in $\M$ and the other is a morphism in $sSet$, the category of simplicial sets. For a topological model category the same axiom is used again but with one of the maps in $\M$ and one in $Top$. 
We refer the reader to Definition 4.2.18 in \cite{hovey-book} for details.

We will at times also need to assume that \textit{cofibrant objects are flat} in $\M$, i.e. that whenever $X$ is cofibrant and $f$ is a weak equivalence then $f\otimes X$ is a weak equivalence. When a monoidal model category satisfies this condition, it is called a {\em tensor model category} in \cite{fausk-isaksen} (Section 12). Finally, we remind the reader of the monoid axiom of Definition 3.3 in \cite{SS00}.

Given a class of maps $\C$ in $\M$, let $\C \otimes \M$ denote the class of maps $f \otimes id_X$ where $f\in \C$ and $X\in \M$. A model category is said to satisfy the \textit{monoid axiom} if every map in (Trivial-Cofibrations $\otimes \M)$-cell is a weak equivalence.

We will be discussing preservation of algebraic structure as encoded by an operad. Let $P$ be an operad valued in $\M$ (for a general discussion of the interplay between operads and homotopy theory see \cite{BM03}). Let $P$-alg($\M$) denote the category whose objects are $P$-algebras in $\M$ (i.e. admit an action of $P$) and whose morphisms are $P$-algebra homomorphisms (i.e. respect the $P$-action). The free $P$-algebra functor from $\M$ to $P$-alg$(\M)$ is left adjoint to the forgetful functor. We will say that $P$-alg$(\M)$ \textit{inherits} a model structure from $\M$ if the model structure is transferred across this adjunction, i.e. if a $P$-algebra homomorphism is a weak equivalence (resp. fibration) if and only if it is so in $\M$. In Section 4 of \cite{BM03}, an operad $P$ is said to be \textit{admissible} if $P$-alg$(\M)$ inherits a model structure in this way.

Finally, we remind the reader about the process of Bousfield localization as discussed in \cite{hirschhorn}. This is a general machine that starts with a (nice) model category $\M$ and a set of morphisms $\C$ and produces a new model structure $L_\C(\M)$ on the same category in which maps in $\C$ are now weak equivalences. Furthermore, this is done in a universal way, introducing the smallest number of new weak equivalences as possible. When we say Bousfield localization we will always mean left Bousfield localization. So the cofibrations in $L_\C(\M)$ will be the same as the cofibrations in $\M$.

Bousfield localization proceeds by first constructing the fibrant objects of $L_\C(\M)$ and then constructing the weak equivalences. In both cases this is done via simplicial mapping spaces $\map(-,-)$. If $\M$ is a simplicial or topological model category then one can use the hom-object in $sSet$ or $Top$. Otherwise a framing is required to construct the simplicial mapping space. We refer the reader to \cite{hovey-book} or \cite{hirschhorn} for details on this process.

An object $N$ is said to be \textit{$\C$-local} if it is fibrant in $\M$ and if for all $g:X\to Y$ in $\C$, $\map(g,N):\map(Y,N)\to \map(X,N)$ is a weak equivalence in $sSet$. These objects are precisely the fibrant objects in $L_\C(\M)$. A map $f:A\to B$ is a \textit{$\C$-local equivalence} if for all $N$ as above, $\map(f,N):\map(B,N)\to \map(A,N)$ is a weak equivalence. These maps are precisely the weak equivalences in $L_\C(\M)$.

It is often more convenient to work with left Bousfield localizations that invert a set of cofibrations (i.e. with left derived Bousfield localization). This can always be guaranteed in the following way. For any map $f$ let $Qf$ denote the cofibrant replacement and let $\tilde{f}$ denote the left factor in the cofibration-trivial fibration factorization of $Qf$. Then $\tilde{f}$ is a cofibration between cofibrant objects and we may define $\tilde{\C} = \{\tilde{f}\;|\; f\in \C\}$. Localization with respect to $\tilde{\C}$ yields the same result as localization with respect to $\C$, so our assumption that the maps in $\C$ are cofibrations between cofibrant objects loses no generality. We thus make the following convention.

\begin{convention} \label{convention:C-all-cofibrations}
Throughout this paper we assume $\C$ is a set of cofibrations between cofibrant objects, and that the model category $L_\C(\M)$ exists.
\end{convention}

The existence of $L_\C(\M)$ can be guaranteed by assuming $\M$ is left proper and either combinatorial (as discussed in \cite{barwickSemi}) or cellular (as discussed in \cite{hirschhorn}). A model category is \textit{left proper} if pushouts of weak equivalences along cofibrations are again weak equivalences. We will make this a standing hypothesis on $\M$. However, as we have not needed the cellularity or combinatoriality assumptions for our work, outside of the existence of $L_\C(\M)$, we have decided not to assume them. In this way if a Bousfield localization is known to exist for some reason other than the theory in \cite{hirschhorn} then our results will be applicable.

\section{General Preservation Result}
\label{sec:main-preservation}

In this section we provide a general result regarding when Bousfield localization preserves $P$-algebras. We must first provide a precise definition for this concept. Throughout this section, let $\M$ be a monoidal model category and let $\C$ be a class of maps in $\M$ such that Bousfield localization $L_\C(\M)$ is a also monoidal model category. On the model category level the functor $L_\C$ is the identity. So when we write $L_\C$ as a functor we shall mean the composition of derived functors $\Ho(\M)\to \Ho(L_\C(\M)) \to \Ho(\M)$, i.e. $E\to L_\C(E)$ is the unit map of the adjunction $\Ho(\M) \leftrightarrows \Ho(L_\C(\M))$. In particular, for any $E$ in $\M$, $L_\C(E)$ is weakly equivalent to $R_\C Q E$ where $R_\C$ is a choice of fibrant replacement in $L_\C(\M)$ and $Q$ is a cofibrant replacement in $\M$. 

Let $P$ be an operad valued in $\M$. Because the objects of $L_\C(\M)$ are the same as the objects of $\M$, $P$ is also valued in $L_\C(\M)$. Thus, we may consider $P$-algebras in both categories and these classes of objects agree (because the $P$-algebra action is independent of the model structure). We denote the categories of $P$-algebras by $P$-alg($\M)$ and $P$-alg$(L_\C(\M)$). These are identical as categories, but in a moment they will receive different model structures.

\begin{defn}
Assume that $\M$ and $L_\C(\M)$ are monoidal model categories. Then $L_\C$ is said to \textit{preserve $P$-algebras} if 

\begin{enumerate}
\item When $E$ is a $P$-algebra there is some $P$-algebra $\tilde{E}$ that is weakly equivalent in $\M$ to $L_\C(E)$.

\item In addition, when $E$ is a cofibrant $P$-algebra, then there is a choice of $\tilde{E}$ in $P$-alg$(\M)$ with $U(\tilde{E})$ local in $\M$, there is a $P$-algebra homomorphism $r_E:E\to \tilde{E}$ that lifts the localization map $l_E:E\to L_\C(E)$ up to homotopy, and there is a weak equivalence $\beta_E:L_\C(U(E))\to U(\tilde{E})$ such that $\beta_E \circ l_{UE} \cong Ur_E$ in $\ho(\M)$.
\end{enumerate}
\end{defn}

The notion of preservation was also considered in \cite{CGMV}, but only for cofibrant $E$. Recall that when we say $P$-alg$(\M)$ \textit{inherits} a model structure from $\M$ we mean that this model structure is transferred by the free-forgetful adjunction. In particular, a map of $P$-algebras $f$ is a weak equivalence (resp. fibration) if and only if $f$ is a weak equivalence (resp. fibration) in $\M$. 

\begin{theorem} \label{bigthm}
Let $\M$ be a monoidal model category such that the Bousfield localization $L_\C(\M)$ exists and is a monoidal model category. Let $P$ be an operad valued in $\M$. If the categories of $P$-algebras in $\M$ and in $L_\C(\M)$ inherit model structures from $\M$ and $L_\C(\M)$ then $L_\C$ preserves $P$-algebras.
\end{theorem}

\begin{proof}
Let $R_\C$ denote fibrant replacement in $L_\C(\M)$, let $R_{\C,P}$ denote fibrant replacement in $P$-alg$(L_\C(\M))$, and let $Q_P$ denote cofibrant replacement in $P$-alg$(\M)$. We will prove the first form of preservation and our method of proof will allow us to deduce the second form of preservation in the special case where $E$ is a cofibrant $P$-algebra.

Let $E$ be a $P$-algebra, and define $\tilde{E} = R_{\C,P}Q_P(E)$. Because $Q$ is the left derived functor of the identity adjunction between $\M$ and $L_\C(\M)$, and $R_\C$ is the right derived functor of the identity, we know that $L_\C(E) \simeq R_\C Q(E)$. We must therefore show $R_\C Q(E)\simeq R_{\C,P}Q_P(E)$.

The map $Q_P E\to E$ is a trivial fibration in $P$-alg$(\M)$, hence in $\M$. The map $QE\to E$ is also a weak equivalence in $\M$. Consider the following lifting diagram in $\M$:
\begin{align} \label{diagram-cof-rep-lift}
\xymatrix{\emptyset \ar[r] \ar@{^{(}->}[d] & Q_P E \ar@{->>}[d]^\simeq \\
QE \ar[r] \ar@{..>}[ur] & E}
\end{align}

The lifting axiom gives the map $QE\to Q_P E$ and it is necessarily a weak equivalence in $\M$ by the 2 out of 3 property.

Since $Q_P E$ is a $P$-algebra in $\M$ it must also be a $P$-algebra in $L_\C(\M)$, since the monoidal structure of the two categories is the same. We may therefore construct a lift:
\begin{align} \label{diagram:RCQPE}
\xymatrix{Q_P E \ar@{^{(}->}[d] \ar[r] & R_{\C,P} Q_P E \ar@{->>}[d]\\
R_\C Q_P E \ar[r] \ar@{..>}[ur] & \ast}
\end{align}

In this diagram the left vertical map is a weak equivalence in $L_\C(\M)$ and the top horizontal map is a weak equivalence in $P$-alg$(L_\C(\M))$. Because the model category $P$-alg$(L_\C(\M))$ inherits weak equivalences from $L_\C(\M)$, this map is a weak equivalence in $L_\C(\M)$. Therefore, by the 2 out of 3 property, the lift is a weak equivalence in $L_\C(\M)$. We make use of this map as the horizontal map in the lower right corner of the diagram below.

The top horizontal map $QE \to Q_P E$ in the following diagram is the first map we constructed, which was proven to be a weak equivalence in $\M$. The square in the diagram below is then obtained by applying $R_\C$ to that map. In particular, $R_\C QE \to R_\C Q_P E$ is a weak equivalence in $L_\C(\M)$:
\begin{align} \label{diagram:finish-main}
\xymatrix{QE \ar[r] \ar[d] & Q_P E\ar[d] & \\ R_\C QE \ar[r] & R_\C Q_P E \ar[r] & R_{\C,P}Q_P E}
\end{align}

We have shown that both of the bottom horizontal maps are weak equivalences in $L_\C(\M)$. Thus, by the 2 out of 3 property, their composite $R_\C QE \to R_{\C,P}Q_PE$ is a weak equivalence in $L_\C(\M)$. All the objects in the bottom row are fibrant in $L_\C(\M)$, so these $\C$-local equivalences are actually weak equivalences in $\M$.

As $E$ was a $P$-algebra and $Q_P$ and $R_{\C,P}$ are endofunctors on categories of $P$-algebras, it is clear that $R_{\C,P}Q_PE$ is a $P$-algebra. We have just shown that $L_\C(E)$ is weakly equivalent to this $P$-algebra, so we are done.

We turn now to the case where $E$ is assumed to be a cofibrant $P$-algebra. We have seen that there is an $\M$-weak equivalence $R_\C QE \to R_{\C,P}Q_P E$, and above we took $R_{\C,P}Q_P E$ in $\M$ as our representative for $L_\C(E)$ in $\Ho(\M)$. Because $E$ is a cofibrant $P$-algebra, there are weak equivalences $E \leftrightarrows Q_P(E)$ in $P$-alg$(L_\C(\M))$. This is because all cofibrant replacements of a given object are weakly equivalent, e.g. by diagram (\ref{diagram-cof-rep-lift}). So passage to $Q_P(E)$ is unnecessary when $E$ is cofibrant, and we take $\tilde{E} = R_{\C,P} E$ as our representative for $L_\C(E)$. Observe that $U(\tilde{E})$ is local because the model structure on $P$-algebras is transferred. The $P$-algebra morphism $r_E:E\to \tilde{E}$ is just the fibrant replacement map $R_{\cat{C},P}$, and lifts the localization map $E\to L_\C(E)$ in $\Ho(\M)$. The comparison $\beta_E$ is the following lift in $L_\C(\M)$:

\begin{align} \label{diagram:defining-beta}
\xymatrix{UE \ar@{^(->}[d]_{\simeq_\cat{C}} \ar[r] & U\tilde{E} \ar@{->>}[d] \\
L_\cat{C}(UE) \ar[r] \ar@{..>}[ur]_{\beta_E}& \ast}
\end{align}
The two out of three property guarantees that $\beta_E$ is a weak equivalence (again using that the model structure on $P$-algebras is transferred), and the diagram above demonstrates that $\beta_E \circ l_{UE} \cong Ur_E$ in $\ho(\M)$.
\end{proof}

This theorem alone would not be a satisfactory answer to the question of when $L_\C$ preserves $P$-algebras, because there is no clear way to check the hypotheses. For this reason, in the coming sections we will discuss conditions on $\M$ and $P$ so that $P$-algebras inherit model structures, and then we will discuss which localizations $L_\C$ preserve these conditions (so that $P$-alg$(L_\C(\M))$ inherits a model structure from $L_\C(\M)$). One such condition on $\M$ is the monoid axiom. In Section \ref{sec:localization-and-monoid-ax}, we discuss which localizations $L_\C$ preserve the monoid axiom. However, it will turn out that the monoid axiom is not necessary in order for our preservation results to apply. This is because the work in \cite{hovey-monoidal} and \cite{spitzweck-thesis} produces semi-model structures on $P$-algebras and these will be enough for our proof above to go through.

Observe that in the proof above we only used formal properties of fibrant and cofibrant replacement functors, and the fact that the model structures on $P$-algebras were inherited from $\M$ and $L_\C(\M)$. So it should not come as a surprise to experts that the same proof works when $P$-algebras only form semi-model categories. For completeness, we remind the reader of the definition of a semi-model category. The motivating example is when $\D$ is obtained from $\M$ via the general transfer principle for transferring a model structure across an adjunction (see Lemma 2.3 in \cite{SS00} or Theorem 12.1.4 in \cite{fresse-book}) when not all the conditions needed to get a full model structure are satisfied.

In particular, the reader should imagine that weak equivalences and fibrations in $\D$ are maps that forget to weak equivalences and fibrations in $\M$, and that the generating (trivial) cofibrations of $\D$ are maps of the form $F(I)$ and $F(J)$ where $F:\M \to \D$ is the free algebra functor and $I$ and $J$ are the generating (trivial) cofibrations of $\M$. The following is Definition 1 from \cite{spitzweck-thesis} and Definition 12.1.1 in \cite{fresse-book}. Cofibrant should be taken to mean cofibrant in $\D$.

\begin{defn}

A \textit{semi-model category} is a bicomplete category $\D$, an adjunction $F:\M \leftrightarrows \D:U$ where $\M$ is a model category, and subcategories of weak equivalences, fibrations, and cofibrations in $\D$ satisfying the following axioms:

\begin{enumerate}
\item $U$ preserves fibrations and trivial fibrations.
\item $\D$ satisfies the two out of three axiom and the retract axiom.
\item Cofibrations in $\D$ have the left lifting property with respect to trivial fibrations. Trivial cofibrations in $\D$ whose domain is cofibrant have the left lifting property with respect to fibrations.
\item Every map in $\D$ can be functorially factored into a cofibration followed by a trivial fibration. Every map in $\D$ whose domain is cofibrant can be functorially factored into a trivial cofibration followed by a fibration.
\item The initial object in $\D$ is cofibrant.
\item Fibrations and trivial fibrations are closed under pullback.
\end{enumerate}

$\D$ is said to be \textit{cofibrantly generated} if there are sets of morphisms $I'$ and $J'$ in $\D$ such that $I'$-inj is the class of trivial fibrations and $J'$-inj the class of fibrations in $\D$, if the domains of $I'$ are small relative to $I'$-cell, and if the domains of $J'$ are small relative to maps in $J'$-cell whose domain becomes cofibrant in $\M$.
\end{defn}

Note that the only difference between a semi-model structure and a model structure is that one of the lifting properties and one of the factorization properties requires the domain of the map in question to be cofibrant. Because fibrant and cofibrant replacements are constructed via factorization, (4) implies that every object has a cofibrant replacement and that objects with cofibrant domain have fibrant replacements. So one could construct a fibrant replacement functor that first does cofibrant replacement and then does fibrant replacement. These functors behave as they would in the presence of a full model structure. 

We are now prepared to state our preservation result in the presence of only a semi-model structure on $P$-algebras. Again, when we say $P$-algebras inherit a semi-model structure we mean with weak equivalences and fibrations reflected and preserved by the forgetful functor.

\begin{corollary} \label{cor:loc-pres-alg-for-semi}
Let $\M$ be a monoidal model category such that the Bousfield localization $L_\C(\M)$ exists and is a monoidal model category. Let $P$ be an operad valued in $\M$. If the subcategories of $P$-algebras in $\M$ and in $L_\C(\M)$ inherit semi-model structures from $\M$ and $L_\C(\M)$ then $L_\C$ preserves $P$-algebras.
\end{corollary}

\begin{proof}

The proof proceeds exactly as the proof of Theorem \ref{bigthm}. We highlight where care must be taken in the presence of semi-model categories. As remarked above, the cofibrant replacement $Q_P$ in the semi-model category $P$-alg$(\M)$ exists and the cofibrant replacement map $Q_P E \to E$ is a weak equivalence in $P$-alg($\M)$, hence in $\M$, because the semi-model structure is transferred. Diagram (\ref{diagram-cof-rep-lift}) is a lifting diagram in $\M$, so still yields a weak equivalence $QE\to Q_P E$. 

Next, the fibrant replacement $R_\C Q_P E$ is a replacement in the model category $L_\C(\M)$. The fibrant replacement $Q_P E \to R_{\C,P}Q_P E$ is a fibrant replacement in the semi-model category $P$-alg$(L_\C(\M))$, and exists because the domain of $Q_P E \to \ast$ is cofibrant in $P$-alg$(L_\C(\M))$. The resulting object $R_{\C,P}Q_P E$ is fibrant in $P$-alg$(L_\C(\M))$ hence in $L_\C(\M)$, since the semi-model structure is transferred. The lift in (\ref{diagram:RCQPE}) is a lift in $L_\C(\M)$, and again by the two out of three property in $L_\C(\M)$ the diagonal map is a $\C$-local equivalence:
\begin{align*}
\xymatrix{Q_P E \ar@{^{(}->}[d] \ar[r] & R_{\C,P} Q_P E \ar@{->>}[d]\\
R_\C Q_P E \ar[r] \ar@{..>}[ur] & \ast}
\end{align*}

Next, the map $R_\C Q E\to R_\C Q_P E$ in (\ref{diagram:finish-main}) is fibrant replacement in the model category $L_\C(\M)$, and so the argument that $R_\C QE \to R_\C Q_P E$ is a $\C$-local equivalence remains unchanged. 
\begin{align*}
\xymatrix{QE \ar[r] \ar[d] & Q_P E\ar[d] & \\ R_\C QE \ar[r] & R_\C Q_P E \ar[r] & R_{\C,P}Q_P E}
\end{align*}

The composite across the bottom $R_\C QE \to R_{\C,P} Q_P E$ is a weak equivalence between fibrant objects in $L_\C(\M)$ and so is a weak equivalence in $\M$, as in the proof of the theorem.

Finally, for the case of $E$ cofibrant in the semi-model category $P$-alg$(\M)$, note that the localization map $E\to L_\C(E)$ is again fibrant replacement $E \to R_{\C,P} E$ in $P$-alg($L_\C(\M))$. This exists because the domain is cofibrant by assumption. By construction, this map is a $P$-algebra morphism, as desired. The lift defining $\beta$ in (\ref{diagram:defining-beta}) occurs in $L_\C(\M)$, and the rest of the proof only uses that weak equivalences and fibrations in $P$-alg($L_\C(\M))$ forget to weak equivalences and fibrations in $L_\C(\M)$.

\end{proof}

\begin{remark}
Corollary \ref{cor:loc-pres-alg-for-semi} has been generalized to algebras over colored operads in \cite{white-yau}, and to right Bousfield localization in \cite{white-yau-coloc}. It has been applied to localizations of Smith ideals in \cite{white-yau6}.
\end{remark}

\section{Monoidal Bousfield Localizations}
\label{sec:monoidal-bous-loc}

In both Theorem \ref{bigthm} and Corollary \ref{cor:loc-pres-alg-for-semi} we assumed that $L_\C(\M)$ is a monoidal model category. In this section we provide conditions on $\M$ and $\C$ so that this occurs. First, we provide an example demonstrating that the pushout product axiom can fail for $L_\C(\M)$, even if it holds for $\M$.

\begin{example} \label{example:ppAxiom-fails-stmod}
It is not true that every Bousfield localization of a monoidal model category is a monoidal model category. Let $R = \F_2[\Sigma_3]$. An $R$ module is simply an $\F_2$ vector space with an action of the symmetric group $\Sigma_3$. Because $R$ is a Frobenius ring, we may pass from $R$-mod to the \textit{stable module category} $StMod(R)$ by identifying any two morphisms whose difference factors through a projective module. 

Section 2.2 of \cite{hovey-book} introduces a model category $\M$ of $R$-modules whose homotopy category is $StMod(R)$. Furthermore, a series of propositions in Section 2.2 demonstrate that $\M$ is a finitely generated, combinatorial, stable model category in which all objects are cofibrant (hence, $\M$ is also left proper). Proposition 4.2.15 of \cite{hovey-book} proves that for $R=\F_2[\Sigma_3]$, this model category is a monoidal model category because $R$ is a Hopf algebra over $\F_2$. The monoidal product of two $R$-modules is $M\otimes_{\F_2}N$ where $R$ acts via its diagonal $R \to R\otimes_{\F_2}R$. 

We now check that cofibrant objects are flat in $\M$. By the pushout product axiom, $X\otimes -$ is left Quillen. Since all objects are cofibrant, all weak equivalences are weak equivalences between cofibrant objects. So Ken Brown's lemma implies $X\otimes -$ preserves weak equivalences.

Let $f:0\to \F_2$, where the codomain has the trivial $\Sigma_3$ action. We'll show that the Bousfield localization with respect to $f$ cannot be a monoidal Bousfield localization. First observe that being $f$-locally trivial is equivalent to having no $\Sigma_3$-fixed points, and this is equivalent to failing to admit $\Sigma_3$-equivariant maps from $\F_2$ (the non-identity element would need to be taken to a $\Sigma_3$-fixed point because the $\Sigma_3$-action on $\F_2$ is trivial).

If the pushout product axiom held in $L_f(\M)$ then the pushout product of two $f$-locally trivial cofibrations $g,h$ would have to be $f$-locally trivial. We will now demonstrate an $f$-locally trivial object $N$ for which $N\otimes_{\F_2} N$ is not $f$-locally trivial, so $(\emptyset \to N) \boxprod (\emptyset \to N)$ is not a trivial cofibration in $L_f(\M)$.

Define $N \cong \F_2 \oplus \F_2$ where the element $(12)$ sends $a=(1,0)$ to $b=(0,1)$ and the element $(123)$ sends $a$ to $b$ and $b$ to $c=a+b$. The reader can check that $(12)(123)$ acts the same as $(123)^2(12)$, so that this is a well-defined $\Sigma_3$-action. This object $N$ is $f$-locally trivial. It does not admit any maps from $\F_2$ because it has no $\Sigma_3$-fixed points. However, $N\otimes_{\F_2} N$ is not $f$-locally trivial because $N\otimes_{\F_2} N$ does admit any map from $\F_2$ taking the non-identity element of $\F_2$ to the $\Sigma_3$-invariant element $a\otimes a + b\otimes b + c\otimes c$. Thus, $L_f(\M)$ is not a monoidal model category.

\end{example}

In order to get around examples such as the above we must place hypotheses on the maps $\C$ that we are inverting. A similar program was conducted in \cite{CGMV}, where localizations of stable model categories were assumed to commute with suspension. Similarly, a condition on a stable localization to ensure that it is additionally monoidal was given in Definition 5.2 of \cite{Barnes12LeftAndRight} and the same condition appeared in Theorem 4.46 of \cite{barwickSemi}. This condition states that $\C \boxprod I$ is contained in the $\C$-local equivalences.

\begin{remark}
The counterexample above fails to satisfy the condition that $\C \boxprod I$ is contained in the $\C$-local equivalences. If this condition were satisfied then $I$ would be contained in the $f$-local equivalences and this would imply all cofibrant objects (hence all objects) are $f$-locally trivial. But $0\to N\otimes_{\F_2} N$ is not $f$-locally trivial. Thus, this counterexample has no bearing on the work of \cite{Barnes12LeftAndRight} or \cite{barwickSemi}.
\end{remark}

\begin{remark}
The counterexample demonstrates a general principle that we now highlight. In any $G$-equivariant world, there are multiple spheres due to the different group actions. In the example above, one can suspend by representations of $\Sigma_n$, i.e. copies of $\F_2$ on which $\Sigma_n$ acts. The 1-point compactification of such an object is a sphere $S^n$ on which $\Sigma_n$ acts. A localization that kills a representation sphere should not be expected to respect the monoidal structure, because not all acyclic cofibrant objects can be built from one of the representation spheres alone. In particular, $N\otimes N$ will not be in the smallest thick subcategory generated by $\F_2$. The point is that the homotopy categories of stable model categories in an equivariant context are not monogenic axiomatic stable homotopy categories in the sense of \cite{hovey-palmieri-strickland}.

Note that this example also demonstrates that the monoid axiom can fail on $L_\C(\M)$. The author does not know an example of a model category satisfying the pushout product axiom but failing the monoid axiom.
\end{remark}

In our applications we will need to know that $L_\C(\M)$ satisfies the pushout product axiom, the unit axiom, and the axiom that cofibrant objects are flat. We therefore give a name to such localizations, and then we characterize them. The reader is advised to keep Convention \ref{convention:C-all-cofibrations} in mind.

\begin{defn} \label{defn-monoidal-localization}
A Bousfield localization $L_\C$ is said to be a \textit{monoidal Bousfield localization} if $L_\C(\M)$ satisfies the pushout product axiom, the unit axiom, and the axiom that cofibrant objects are flat.
\end{defn}

\begin{theorem} \label{thm:PPAxiom}
Suppose that $\M$ is a cofibrantly generated monoidal model category in which cofibrant objects are flat and the domains of the generating cofibrations are cofibrant. Let $I$ denote the generating cofibrations of $\M$. Then $L_\C$ is a monoidal Bousfield localization if and only if every map of the form $f \otimes id_K$, where $f$ is in $\C$ and $K$ is a domain or codomain of a map in $I$, is a $\C$-local equivalence.
\end{theorem}

\begin{theorem} \label{thm:PPAxiom-nontractable}
Suppose $\M$ is a cofibrantly generated monoidal model category in which cofibrant objects are flat. Then $L_\C$ is a monoidal Bousfield localization if and only if every map of the form $f \otimes id_K$, where $f$ is in $\C$ and $K$ is cofibrant, is a $\C$-local equivalence.
\end{theorem}

Note that the condition $\C \boxprod I \subset \C$-local equivalences, from \cite{Barnes12LeftAndRight, barwickSemi}, implies the condition from these theorems. In fact, one can prove it is equivalent to $L_\C(\M)$ being a monoidal model category, because $\C$ can be taken to be a set of $\C$-local trivial cofibrations. However, the condition stated in the theorems above is easier to check. We shall prove Theorem \ref{thm:PPAxiom} in Subsection \ref{subsec:proof-tractable} and we shall prove Theorem \ref{thm:PPAxiom-nontractable} in Subsection \ref{subsec:proof-nontractable}. These theorems demonstrate precisely what must be done if one wishes to invert a given set of morphisms $\C$ and ensure that the resulting model structure is a monoidal model structure. 

\begin{defn}
Suppose $\M$ is left proper, is either cellular or combinatorial, and that the domains of the generating cofibrations are cofibrant. The \textit{smallest monoidal Bousfield localization} which inverts a given set of morphisms $\C$ is the Bousfield localization with respect to the set $\C' = \{\C \otimes id_K\}$ where $K$ runs through the domains and codomains of the generating cofibrations $I$.
\end{defn}

This notion has already been used in \cite{hovey-white}. The reason for the hypothesis on the domains of the generating cofibrations is to ensure that $\C'$ is a set. Requiring left properness and either cellularity or combinatoriality ensures that $L_{\C'}$ exists. The smallest Bousfield localization has a universal property, that we now highlight.

\begin{proposition}
Suppose $\C'$ is the smallest monoidal Bousfield localization inverting $\C$, and let $j:\M \to L_{\C'}(\M)$ be the left Quillen functor realizing the localization. Suppose $\cat{N}$ is a monoidal model category with cofibrant objects flat. Suppose $F:\M \to \cat{N}$ is a monoidal left Quillen functor such that $\mathbb{L}F$ takes the images of $\C$ in $\ho(\M)$ to isomorphisms in $\ho(\cat{N})$. Then there is a unique monoidal left Quillen functor $\delta:L_{\C'}\M \to \cat{N}$ such that $\delta j = F$.
\end{proposition}

\begin{proof}
Suppose $F:\M \to \cat{N}$ is a monoidal left Quillen functor, that $\cat{N}$ has cofibrant objects flat, and that $\mathbb{L}F$ takes the images of $\C$ in $\ho(\M)$ to isomorphisms in $\ho(\cat{N})$. Then $F$ also takes the images of maps in $\C'$ to isomorphisms in $\ho(\cat{N})$, because for any $f\in \C$ and any cofibrant $K$, $F(f\otimes K) \cong F(f)\otimes F(K)$ is a weak equivalence in $\cat{N}$. This is because $F(K)$ is cofibrant in $\cat{N}$ (as $F$ is left Quillen), cofibrant objects are flat in $\cat{N}$, and $F(f)$ is a weak equivalence in $\cat{N}$ by hypothesis.

The universal property of the localization $L_{\C'}$ then provides a unique left Quillen functor $\delta:L_{\C'}\M \to \cat{N}$ that is the same as $F$ on objects and morphisms (Theorem 3.3.18 and Theorem 3.3.19 in \cite{hirschhorn}). In particular, $\delta$ is a monoidal functor and $\delta q = Fq:F(QS)\to F(S)$ is a weak equivalence in $\cat{N}$ because the cofibrant replacement $QS\to S$ is the same in $L_{\C'}(\M)$ as in $\M$. So $\delta$ is a unique monoidal left Quillen functor as required, and the commutativity $\delta j = F$ follows immediately from the definition of $\delta$.
\end{proof}

\subsection{Proof of Theorem \ref{thm:PPAxiom}} \label{subsec:proof-tractable}

In this section we will prove Theorem \ref{thm:PPAxiom}. We first prove that under the hypotheses of Theorem \ref{thm:PPAxiom}, cofibrant objects are flat in $L_\C(\M)$.

\begin{proposition} \label{cofObjFlat}
Let $\M$ be a cofibrantly generated monoidal model category in which cofibrant objects are flat and the domains of the generating cofibrations are cofibrant. Let $I$ denote the generating cofibrations of $\M$. Suppose that every map of the form $f \otimes id_K$, where $f$ is in $\C$ and $K$ is a domain or codomain of a map in $I$, is a $\C$-local equivalence. Then cofibrant objects are flat in $L_\C(\M)$.
\end{proposition}

\begin{proof}

We must prove that the class of maps $\{g\otimes X \;|\; g$ is a $\C$-local equivalence and $X$ is a cofibrant object$\}$ is contained in the $\C$-local equivalences. Let $X$ be a cofibrant object in $L_\C(\M)$ (equivalently, in $\M$). Let $g:A\to B$ be a $\C$-local equivalence. To prove $-\otimes X$ preserves $\C$-local equivalences, it suffices to show that it takes $L_\C(\M)$ trivial cofibrations between cofibrant objects to weak equivalences. This is because we can always do cofibrant replacement on $g$ to get $Qg:QA\to QB$. While $Qg$ need not be a cofibration in general, we can always factor it into $QA\hookrightarrow Z \stackrel{\simeq}{\twoheadrightarrow} QB$. By abuse of notation we will continue to use the symbol $QB$ to denote $Z$, and we will rename the cofibration $QA\to Z$ as $Qg$ since $Z$ is cofibrant and maps via a trivial fibration to $B$. Smashing with $X$ gives:
\begin{align*}
\xymatrix{QA\otimes X \ar[r] \ar[d] & QB \otimes X \ar[d] \\ A\otimes X \ar[r] & B\otimes X}
\end{align*}

If we prove that $Qg\otimes X$ is a $\C$-local equivalence, then $g\otimes X$ must also be by the two out of three property, since the vertical maps are weak equivalences in $\M$ due to $X$ being cofibrant and cofibrant objects being flat in $M$. So we may assume that $g$ is an $L_\C(\M)$ trivial cofibration between cofibrant objects. Since $X$ is built as a transfinite composition of pushouts of maps in $I$, we proceed by transfinite induction. For the rest of the proof, let $K, K_1,$ and $K_2$ denote domains/codomains of maps in $I$. These objects are cofibrant in $\M$ by hypothesis, so they are also cofibrant in $L_\C(\M)$.

For the base case $X=K$ we appeal to Theorem 3.3.18 in \cite{hirschhorn}. The composition $F = id \circ K\otimes -: \M \to \M \to L_\C(\M)$ is left Quillen because $K$ is cofibrant. $F$ takes maps in $\C$ to weak equivalences by hypothesis. So Theorem 3.3.18 implies $F$ induces a left Quillen functor $K\otimes -:L_\C(\M)\to L_\C(\M)$. Thus, $K\otimes -$ takes $\C$-local equivalences between cofibrant objects to $\C$-local equivalences and in particular takes $Qg$ to a $\C$-local equivalence. Note that this is the key place in this proof where we use the hypothesis that $L_\C$ is a monoidal Bousfield localization. This theorem is the primary tool when one wishes to get from a statement about $\C$ to a statement about all $\C$-local equivalences.

For the successor case, suppose $X_{\alpha}$ is built from $K$ as above and is flat in $L_\C(\M)$. Suppose $X_{\alpha+1}$ is built from $X_\alpha$ and a map in $I$ via a pushout diagram:
\begin{align*}
\xymatrix{K_1\ar@{^{(}->}[r]^i \ar[d] \ar@{}[dr]|(.7){\Searrow} & K_2 \ar[d] \\ X_\alpha \ar[r] & X_{\alpha+1}}
\end{align*}

We smash this diagram with $g: A\to B$ and note that smashing a pushout square with an object yields a pushout square.

\begin{align*}
\xymatrix{A\otimes K_1 \ar[rr]^{A\otimes i} \ar[dd] \ar[dr]^{g\otimes K_1} & & A\otimes K_2 \ar[dr]^{g\otimes K_2} \ar[dd] & \\ & B\otimes K_1 \ar[rr] \ar[dd] & & B\otimes K_2 \ar[dd] \\ A\otimes X_\alpha \ar[rr]^{B\otimes i} \ar[dr]_{g\otimes X_\alpha} & & A\otimes X_{\alpha+1} \ar[dr]^{g\otimes X_{\alpha+1}} & \\ & B\otimes X_\alpha \ar[rr] & & B\otimes X_{\alpha+1}}
\end{align*}

Because $g$ is a cofibration of cofibrant objects, $A$ and $B$ are cofibrant. Because pushouts of cofibrations are cofibrations, $X_\alpha \hookrightarrow X_{\alpha+1}$ for all $\alpha$. Because $X_0$ is cofibrant, $X_\alpha$ is cofibrant for all $\alpha$. So all objects above are cofibrant. Furthermore, $g\otimes K_i = g \boxprod (0\hookrightarrow K_i)$. Thus, by the Pushout Product axiom on $\M$ and the fact that cofibrations in $\M$ match those in $L_\C(\M)$, these maps are cofibrations.

Finally, the maps $g\otimes K_i$ are weak equivalences in $L_\C(\M)$ by the base case above, while $g\otimes X_\alpha$ is a weak equivalence in $L_\C(\M)$ by the inductive hypothesis. Thus, by Dan Kan's Cube Lemma (Lemma 5.2.6 in \cite{hovey-book}), the map $g\otimes X_{\alpha+1}$ is a weak equivalence in $L_\C(\M)$.

For the limit case, suppose we are given a cofibrant object $X = \colim_{\alpha<\beta} X_\alpha$ where each $X_\alpha$ is cofibrant and flat in $L_\C(\M)$. Because each $X_\alpha$ is cofibrant, $g\otimes X_\alpha = g\boxprod (0\hookrightarrow X_\alpha)$ is still a cofibration. By the inductive hypothesis, each $g\otimes X_\alpha$ is also a $\C$-local equivalence, hence a trivial cofibration in $L_\C(\M)$. Since trivial cofibrations are always closed under transfinite composition, $g\otimes X = g\otimes \colim X_\alpha = \colim (g\otimes X_\alpha)$ is also a trivial cofibration in $L_\C(\M)$.
\end{proof}

We now pause for a moment to extract the key point in the proof above, where we applied the universal property of Bousfield localization. This is a reformulation Theorem 3.3.18 in \cite{hirschhorn} that we will need below.

\begin{lemma} \label{lemma:left-Quillen-and-loc}
A left Quillen functor $F:\M\to \M$ induces a left Quillen functor $L_\C F:L_\C(\M) \to L_\C(\M)$ if and only if for all $f\in \C$, $F(f)$ is $\C$-local equivalence.
\end{lemma}

We turn now to the unit axiom.

\begin{proposition} \label{prop:unit-axiom}
If $\M$ satisfies the unit axiom then any Bousfield localization $L_\C(\M)$ satisfies the unit axiom. If cofibrant objects are flat in $\M$ then the map $QS\otimes Y \to Y$, induced by cofibrant replacement $QS\to S$, is a weak equivalence for all $Y$, not just cofibrant $Y$. Furthermore, for any weak equivalence $f:K\to L$ between cofibrant objects, $f\otimes Y$ is a weak equivalence.
\end{proposition}

\begin{proof}

Since $L_\C(\M)$ has the same cofibrations as $\M$, it must also have the same trivial fibrations. Thus, it has the same cofibrant replacement functor and the same cofibrant objects. Thus, the unit axiom on $L_\C(\M)$ follows directly from the unit axiom on $\M$, because a weak equivalence in $\M$ is in particular a $\C$-local equivalence. 

We now assume cofibrant objects are flat and that $Y$ is an object of $\M$. Consider the following diagram:
\begin{align*}
\xymatrix{QS \otimes QY \ar[r] \ar[d] & QY \ar[d] \\ QS\otimes Y \ar[r] & Y}
\end{align*}

The top map is a weak equivalence by the unit axiom for the cofibrant object $QY$. The left vertical map is a weak equivalence because cofibrant objects are flat and $QS$ is cofibrant. The right vertical is a weak equivalence by definition of $QY$. Thus, the bottom arrow is a weak equivalence by the two out of three property.

For the final statement we again apply cofibrant replacement to $Y$ and we get
\begin{align*}
\xymatrix{K \otimes QY \ar[r] \ar[d] & L\otimes QY \ar[d] \\ K\otimes Y \ar[r] & L\otimes Y}
\end{align*}

Again the top horizontal map and the vertical maps are weak equivalences because cofibrant objects are flat (for the first use that $QX$ is cofibrant, for the second use that $K$ and $L$ are cofibrant).
\end{proof}

We turn now to proving Theorem \ref{thm:PPAxiom}. As mentioned in the proof of Proposition \ref{cofObjFlat}, if $h$ and $g$ are $L_\C(\M)$-cofibrations then they are $\M$-cofibrations and so $h\boxprod g$ is a cofibration in $\M$ (hence in $L_\C(\M)$) by the pushout product axiom on $\M$. To verify the rest of the pushout product axiom on $L_\C(\M)$ we must prove that if $h$ is a trivial cofibration in $L_\C(\M)$ and $g$ is a cofibration in $L_\C(\M)$ then $h\boxprod g$ is a weak equivalence in $L_\C(\M)$. 

\begin{proposition} \label{prop:helper-for-monoidal-loc}
Let $\M$ be a cofibrantly generated monoidal model category in which cofibrant objects are flat and the domains of the generating cofibrations are cofibrant. Let $I$ denote the generating cofibrations of $\M$. Suppose that every map of the form $f \otimes id_K$, where $f$ is in $\C$ and $K$ is a domain or codomain of a map in $I$, is a $\C$-local equivalence. Then $L_\C(\M)$ satisfies the pushout product axiom.
\end{proposition}

\begin{proof} 

We have already remarked that the cofibration part of the pushout product axiom on $L_\C(\M)$ follows from the pushout product axiom on $\M$, since the two model categories have the same cofibrations. By Proposition 4.2.5 of \cite{hovey-book} it is sufficient to check the pushout product axiom on generating (trivial) cofibrations. So suppose $h:X\to Y$ is an $L_\C(\M)$ trivial cofibration and $g:K\to L$ is a generating cofibration in $L_\C(\M)$ (equivalently, in $\M$). Then we must show $h\boxprod g$ is an $L_\C(\M)$ trivial cofibration

By hypothesis on $\M$, $K$ and $L$ are cofibrant. Because $h$ is a cofibration, $K\otimes h$ and $L\otimes h$ are cofibrations by the pushout product axiom on $\M$ (because $K\otimes h = (\emptyset \hookrightarrow K)\boxprod h$). By Proposition \ref{cofObjFlat}, cofibrant objects are flat in $L_\C(\M)$. So $K\otimes h$ and $L\otimes h$ are also weak equivalences. In particular, $K\otimes -$ and $L\otimes -$ are left Quillen functors. Consider the following diagram:

\begin{align*}
\xymatrix{K\otimes X \po \ar@{^{(}->}[r]^\simeq \ar[d] & K\otimes Y \ar[d] \ar@/^1pc/[ddr] & \\ L\otimes X \ar[r]^(.3){\simeq} \ar@/_1pc/[drr]_\simeq & (K\otimes Y) \coprod_{K\otimes X} (L\otimes X) \ar[dr]^{h\boxprod g} & \\ & & L\otimes Y}
\end{align*}

The map $L\otimes X \to (K\otimes Y) \coprod_{K\otimes X} (L\otimes X)$ is a trivial cofibration because it is the pushout of a trivial cofibration. Thus, by the two out of three property for the lower triangle, $h\boxprod g$ is a weak equivalence. Since we already knew it was a cofibration (because it is so in $\M$), this means it is a trivial cofibration.
\end{proof}

We are now ready to complete the proof of Theorem \ref{thm:PPAxiom}.

\begin{proof}[Proof of Theorem \ref{thm:PPAxiom}]

We begin with the forwards direction. Suppose $L_\C(\M)$ satisfies the pushout product axiom and has cofibrant objects flat. Let $f$ be any map in $\C$. Note that in particular, $f$ is a $\C$-local equivalence. Because cofibrant objects are flat, the map $f \otimes K$ is a $\C$-local equivalence for any cofibrant $K$. So the collection $\C \otimes K$ is contained in the $\C$-local equivalences, where $K$ runs through the class of cofibrant objects, i.e. $L_\C$ is a monoidal Bousfield localization.

For the converse, we apply our three previous propositions. That cofibrant objects are flat in $L_\C(\M)$ is the content of Proposition \ref{cofObjFlat}. The unit axiom on $L_\C(\M)$ follows from Proposition \ref{prop:unit-axiom} applied to $L_\C(\M)$. 
That the pushout product axiom holds on $L_\C(\M)$ is Proposition \ref{prop:helper-for-monoidal-loc}.
\end{proof}

\subsection{Proof of Theorem \ref{thm:PPAxiom-nontractable}} \label{subsec:proof-nontractable}

We will now prove Theorem \ref{thm:PPAxiom-nontractable}, following the outline above. The proof that cofibrant objects are flat in $L_\C(\M)$ will proceed just as it did in Proposition \ref{cofObjFlat}. Proposition \ref{prop:unit-axiom} again implies the unit axiom in $L_\C(\M)$. Deducing the pushout product axiom on $L_\C(\M)$ will be more complicated without the assumption on the domains of $I$. For this reason, we need the following lemma. First, let $I'$ be obtained from the generating cofibrations $I$ by applying any cofibrant replacement $Q$ to all $i\in I$ and then taking the left factor in the cofibration-trivial fibration factorization of $Qi$. So $I'$ consists of cofibrations between cofibrant objects.

\begin{lemma} \label{lemma:I'UJ}
Suppose $\M$ is a left proper model category cofibrantly generated by sets $I$ and $J$ in which the domains of maps in $J$ are small relative to $I$-cell. Then the sets $I'\cup J$ and $J$ cofibrantly generate $\M$. 
\end{lemma}

\begin{proof}
We verify the conditions given in Definition 11.1.2 of \cite{hirschhorn}. We have not changed $J$, so the fibrations are still precisely the maps satisfying the right lifting property with respect to $J$ and the maps in $J$ still permit the small object argument because the domains are small relative to $J$-cell.

Any map that has the right lifting property with respect to all maps in $I$ is a trivial fibration, so will in particular have the right lifting property with respect to all cofibrations, hence with respect to maps in $I'\cup J$. Conversely, suppose $p$ has the right lifting property with respect to all maps in $I' \cup J$. We are faced with the following lifting problem:
\begin{align*}
\xymatrix{A' \ar[r] \ar[d]^{i'} & A \ar[r] \ar[d]^i & X\ar[d]^p \\
B' \ar[r] & B \ar[r] & Y}
\end{align*}

Because $p$ has lifting with respect to $I'\cup J$, it has the right lifting property with respect to $J$. This guarantees us that $p$ is a fibration. Now because $\M$ is left proper, Proposition 13.2.1 in \cite{hirschhorn} applies to solve the lifting diagram above. In particular, because $p$ has the right lifting property with respect to $I'$, $p$ must have the right lifting property with respect to $I$. Thus, $p$ is a trivial fibration as desired.

We now turn to smallness. Any domain of a map in $J$ is small relative to $J$-cell, but in general this would not imply smallness relative to $I$-cell. We have assumed the domains of maps in $J$ are small relative to $I$-cell, so they are small relative to $(J\cup I')$-cell because $J\cup I'$ is contained in $I$-cell.

Any domain of a map in $I'$ is of the form $QA$ for $A$ a domain of a map in $I$. We will show $QA$ is small relative to $I$-cell. As $J\cup I'$ is contained in $I$-cell this will show $QA$ is small relative to $J\cup I'$. Consider the construction of $QA$ as the left factor in
\begin{align*}
\xymatrix{ & QA \ar@{->>}[dr]^\simeq & \\
\emptyset \ar@{^(->}[ur] \ar[rr] & & A}
\end{align*}

The map $\emptyset \to QA$ is in $I$-cell, so $QA$ is a colimit of cells (let us say $\kappa_A$ many), each of which is $\kappa$-small where $\kappa$ is the regular cardinal associated to $I$ by Proposition 11.2.5 of \cite{hirschhorn}. So for any $\lambda$ greater than the cofinality of $\max(\kappa,\kappa_A)$, a map from $QA$ to a $\lambda$-filtered colimit of maps in $I$-cell must factor through some stage of the colimit because all the cells making up $QA$ will factor in this way. One can find a uniform $\lambda$ for all objects $QA$ by an appeal to Lemma 10.4.6 of \cite{hirschhorn}.

\end{proof}

\begin{remark}
In a combinatorial model category no smallness hypothesis needs to be made because all objects are small. In a cellular model category, the assumption that the domains of $J$ are small relative to cofibrations is included. As these hypotheses are standard when working with left Bousfield localization, we shall say no more about the additional smallness hypothesis placed on $J$ above.
\end{remark}

\begin{corollary}
Suppose $\M$ is a left proper model category cofibrantly generated by sets $I$ and $J$ in which the domains of maps in $J$ are small relative to $I$-cell and are cofibrant. Then there exist a set of generating cofibrations $I'$ with cofibrant domains. 
\end{corollary}

\begin{remark} \label{remark:simpson-example}
Note that this corollary does not say that any left proper, cofibrantly generated model category has generating sets $I$ and $J$ with cofibrant domains. There is an example due to Carlos Simpson (found on page 199 of \cite{simpson-book}) of a left proper, combinatorial model category that has no such sets $I$ and $J$. In this example the cofibrations and trivial cofibrations are the same, so cannot be leveraged against one another in the way we have done above.
\end{remark}

We are now prepared to prove Theorem \ref{thm:PPAxiom-nontractable}.

\begin{proof}[Proof of Theorem \ref{thm:PPAxiom-nontractable}]
First, if $L_\C$ is a monoidal Bousfield localization then every map of the form $f\otimes id_K$, where $f\in \C$ and $K$ is cofibrant, is a $\C$-local equivalence. This is because $f$ is a $\C$-local equivalence and cofibrant objects are flat in $L_\C(\M)$. We turn now to the converse.

Assume every map of the form $f\otimes id_K$, where $f\in \C$ and $K$ is cofibrant, is a $\C$-local equivalence. Then cofibrant objects are flat in $L_\C(\M)$. To see this, let $X$ be cofibrant and define $F(-) = X\otimes -$. Then Lemma \ref{lemma:left-Quillen-and-loc} implies $F$ is left Quillen when viewed as a functor from $L_\C(\M)$ to $L_\C(\M)$. So $F$ takes $\C$-local equivalences between cofibrant objects to $\C$-local equivalences. By the reduction at the beginning of the proof of Proposition \ref{cofObjFlat}, this implies $F$ takes all $\C$-local equivalences to $\C$-local equivalences.

Next, the unit axiom on $L_\C(\M)$ follows from the unit axiom on $\M$, by Proposition \ref{prop:unit-axiom}. Finally, we must prove the pushout product axiom holds on $L_\C(\M)$. As in the proof of Proposition \ref{prop:helper-for-monoidal-loc}, Proposition 4.2.5 of \cite{hovey-book} reduces the problem to checking the pushout product axiom on a set of generating (trivial) cofibrations. We apply Lemma \ref{lemma:I'UJ} to $\M$ and check the pushout product axiom with respect to this set of generating maps. 

As in the case of Theorem \ref{thm:PPAxiom}, let $h:X\to Y$ be a trivial cofibration in $L_\C(\M)$ and let $g:K\to L$ be a generating cofibration. By the lemma, the map $g$ is either a cofibration between cofibrant objects or a trivial cofibration in $\M$. If the former, then the proof of Proposition \ref{prop:helper-for-monoidal-loc} goes through verbatim and proves that $h\boxprod g$ is an $L_\C(\M)$-trivial cofibration, since cofibrant objects are flat in $L_\C(\M)$. If the latter, then because $g$ is a trivial cofibration in $\M$ and $h$ is a cofibration in $\M$ we may apply the pushout product axiom on $\M$ to see that $h\boxprod g$ is a trivial cofibration in $\M$ (hence in $L_\C(\M)$ too). This completes the proof of the pushout product axiom on $L_\C(\M)$.
\end{proof}

\begin{remark}
The use of the lemma demonstrates that this proposition proves something slightly more general. Namely, if $\M$ is cofibrantly generated, left proper, has cofibrant objects flat, and the class of cofibrations is closed under pushout product then $\M$ satisfies the pushout product axiom.

Additionally, one could also prove the forwards direction in the theorem using only that $L_\C(\M)$ satisfies the pushout product axiom. For any cofibrant $K$ we have a cofibration $\phi_K:\emptyset \hookrightarrow K$. Note that for any $f\in \C$, $f\otimes K = f \boxprod \phi_K \subset \C$-local equivalences, because $f$ is a trivial cofibration in $L_\C(\M)$. 
\end{remark}

We record this remark because in the future we hope to better understand the connection between monoidal Bousfield localizations and the closed localizations that appeared in \cite{CGMV}, and this remark may be useful.

\section{Preservation of algebras over $\Sigma$-cofibrant operads} 
\label{sec:preservation-over-cofibrant}

In this section we will provide several applications of the results in the previous section. We remind the reader that for operads valued in $\M$, a map of operads $A\to B$ is said to be a trivial fibration if $A_n \to B_n$ is a trivial fibration in $\M$ for all $n$. An operad $P$ is said to be \textit{cofibrant} if the map from the initial operad into $P$ has the left lifting property in the category of operads with respect to all trivial fibrations of operads. An operad $P$ is said to be \textit{$\Sigma$-cofibrant} if it has this left lifting property only in the category of symmetric sequences. The $E_\infty$-operads considered in \cite{MayGeometry} are $\Sigma$-cofibrant precisely because the $n^{th}$ space is assumed to be an $E\Sigma_n$ space.

We begin with a theorem due to Markus Spitzweck, proven as Theorem 5 in \cite{spitzweck-thesis} and as Theorem A.8 in \cite{gutierrez-rondigs-spitzweck-ostvaer}, that makes it clear that the hypotheses of Corollary \ref{cor:loc-pres-alg-for-semi} are satisfied when $L_\C$ is a monoidal Bousfield localization and $P$ is a cofibrant operad. 

\begin{theorem} \label{thm:spitzweck}
Suppose $P$ is a $\Sigma$-cofibrant operad and $\M$ is a monoidal model category. Then $P$-alg is a semi-model category.
\end{theorem}

This theorem, applied to both $\M$ and $L_\C(\M)$ (if the localization is monoidal), endows the categories of $P$-algebras in $\M$ and $L_\C(\M)$ with inherited semi-model structures. By Corollary \ref{cor:loc-pres-alg-for-semi}, monoidal Bousfield localizations preserve algebras over $\Sigma$-cofibrant operads. In particular, monoidal localizations preserve $A_\infty$ and $E_\infty$-algebras in $\M$, since these algebras are encoded by $A_\infty$ and $E_\infty$-operads $P$ that are $\Sigma$-cofibrant (and weakly equivalent to $Ass$ and $Com$ respectively in the category $Coll(\M)$). When $\M$ is a category of spectra we are free to work with operads valued in spaces because the $\Sigma^\infty$ functor will take a $(\Sigma$-cofibrant) space-valued operad to a $(\Sigma$-cofibrant) spectrum-valued operad with the same algebras.

\subsection{Spaces and Spectra}

We now provide examples demonstrating the power of Theorem \ref{bigthm}. For topological spaces the situation is especially nice. We will always work in the context of pointed spaces, with the Quillen model structure.

\begin{proposition} \label{prop:monoidal-loc-spaces}
Let $\M$ be the model category of (pointed) simplicial sets or $k$-spaces. Every Bousfield localization of $\M$ is a monoidal Bousfield localization.
\end{proposition}

\begin{proof}
For a review of the monoidal model structures on spaces and simplicial sets see Chapter 4 of \cite{hovey-book}. Both are cellular, left proper, monoidal model categories with cofibrant objects flat and the domains of the generating cofibrations cofibrant.

For $\M=sSet$ we can simply rely on Theorem 4.1.1 of \cite{hirschhorn}, which guarantees that $L_\C(\M)$ is a simplicial model category. The pushout product axiom is equivalent to the SM7 axiom for $sSet$, so this proves $L_\C(\M)$ is a monoidal model category and hence that $L_\C$ is monoidal. There is also an elementary proof of this fact, obtained from the proof below by replacing $F(-,-)$ everywhere by $\map(-,-)$.

We turn now to $\M=Top$. By definition, any Bousfield localization $L_\C$ will be a monoidal Bousfield localization as soon as we show $\C \wedge S^n_+$ is contained in the $\C$-local equivalences (the codomains of the generating cofibrations are contractible, so do not matter). As remarked in the discussion below Definition 4.1 in \cite{hovey-white}, for topological model categories Bousfield localization with respect to a set of cofibrations can be defined using topological mapping spaces rather than simplicial mapping spaces (at least when all maps in $\C$ are cofibrations). Let $F(X,Y)$ denote the space of based maps $X\to Y$.

We will make use of Proposition 3.2 in \cite{hovey-spectra}, a version of which states that because $Top$ is left proper and cofibrantly generated, a map $f$ is a weak local equivalence if and only if $F(T,f)$ is a weak equivalence of topological spaces for all $T$ in the (co)domains of the generating cofibrations $I$ in $Top$.

Now consider the following equivalent statements, where $T$ runs through the domains and codomains of generating cofibrations.

$f \mbox{ is a } \C \mbox{-local equivalence} $\\
$\mbox{iff } F(f,Z) \mbox{ is a weak equivalence for all }\C \mbox{-local } Z $\\
$\mbox{iff } F(T,F(f,Z)) \mbox{ is a weak equivalence for all }\C \mbox{-local Z and all } T$ (by Prop. 3.2)  \\
$\mbox{iff } F(T \wedge f,Z) \mbox{ is a weak equivalence for all }\C \mbox{-local } Z $ (by adjointness)\\
$\mbox{iff } T \wedge f \mbox{ is a } \C \mbox{-local equivalence}$

This proves that the class of $\C$-local equivalences is closed under smashing with a domain or codomain of a generating cofibration, so $L_\C$ is a monoidal Bousfield localization.

\end{proof}

The reader may wonder whether all Bousfield localizations preserve algebras over cofibrant operads in general model categories $\M$, i.e. whether all Bousfield localizations are monoidal. This is false, as demonstrated by the following example, from  Section 6 in \cite{CGMV}.

\begin{example} \label{example:postnikov}
Let $\M$ be symmetric spectra, orthogonal spectra, or $\mathbb{S}$-modules. Recall that in topological spaces, the $n^{th}$ Postnikov section functor $P_n$ is the Bousfield localization $L_f$ corresponding to the map $\Sigma f$ where $f:S^n \to *$. Applying $\Sigma^\infty$ gives a map of spectra and we again denote by $P_n$ the Bousfield localization with respect to this map. The Bousfield localization $P_{-1}$ on $\M$ does not preserve $A_\infty$-algebras. If $R$ is a non-connective $A_\infty$-algebra then the unit map $\nu:S\to P_{-1}R$ is null because $\pi_0(P_{-1}R)=0$. Thus, $P_{-1}R$ cannot admit a ring spectrum structure (not even up to homotopy) because $S\wedge P_{-1}R \to P_{-1}R \wedge P_{-1}R \to P_{-1}R$ is not a homotopy equivalence as it would have to be for $P_{-1}R$ to be a homotopy ring. It follows that the model category $P_{-1}\M$ fails the pushout product axiom, because if $P_{-1}\M$ satisfied the pushout product axiom, then $A_\infty$-algebras in  $P_{-1}\M$ would inherit a transferred semi-model structure, by Theorem \ref{thm:spitzweck}, and Corollary \ref{cor:loc-pres-alg-for-semi} would imply that $P_{-1}$ preserves $A_\infty$-algebras.
\end{example}

In \cite{CGMV}, examples of the sort above are prohibited by assuming that $L$-equivalences are closed under the monoidal product. It is then shown in Theorem 6.5 that for symmetric spectra this property is implied if the localization is \textit{stable}, i.e. $L \circ \Sigma \simeq \Sigma \circ L$. We now compare our requirement that $L_\C$ be a monoidal Bousfield localization to existing results regarding preservation of monoidal structure.

\begin{proposition} \label{prop:stable-implies-monoidal}
Let $\M$ be a stable model category. Then every monoidal Bousfield localization is stable. In a monogenic setting such as spectra, every stable localization is monoidal.
\end{proposition}

This is clear, since suspending is the same as smashing with the suspension of the unit sphere. The Postnikov section is clearly not stable, and indeed the counterexample above hinges on the fact that the section has truncated the spectrum by making trivial the degree in which the unit must live. Stable localizations preserve cofiber sequences, but $P_{-1}$ does not. Under the hypothesis that localization respects the monoidal product, Theorem 6.1 of \cite{CGMV} proves that cofibrant algebras over a cofibrant colored operad valued in $sSet_*$ or $Top_*$ are preserved. Theorem \ref{bigthm} recovers this result in the case of operads, and improves on it by extending the class of operads so that they do not need to be valued in $sSet_*$ or $Top_*$, by discussing preservation of non-cofibrant algebras, by weakening the cofibrancy required of the operad to $\Sigma$-cofibrancy (using Theorem \ref{thm:spitzweck} above), and by potentially weakening the hypothesis on the localization. A different generalization of \cite{CGMV} has been given in \cite{gutierrez-rondigs-spitzweck-ostvaer}. 

\begin{proposition}
Every Bousfield localization for which the local equivalences are closed under $\otimes$ is a monoidal Bousfield localization, but the converse fails.
\end{proposition}

\begin{proof}
To see why this fact is true, consider the maps $id_K$ as $L$-equivalences when testing whether or not $id_K\otimes \C$ is a $\C$-local equivalence. To see that the converse fails, take $\C$ to be the generating trivial cofibrations of any cofibrantly generated model category in which the weak equivalences are not closed under $\otimes$. 
\end{proof}

Thus, our hypothesis on a monoidal Bousfield localization is strictly weaker than requiring $L$-equivalences to be closed under $\otimes$. Theorems \ref{thm:PPAxiom} and \ref{thm:PPAxiom-nontractable} demonstrate that the hypothesis that $\C \otimes id_K$ is contained in the $\C$-local equivalences is best-possible, since it $L_\C$ is a monoidal Bousfield localization if and only if this property holds, and without the pushout product axiom on $L_\C(\M)$ the question of preservation of algebras under localization is not even well-posed. Note that cofibrant objects are flat for symmetric spectra by 5.3.10 in \cite{hovey-shipley-smith}.

\begin{remark}
In light of the Postnikov Section example, the argument of Proposition \ref{prop:monoidal-loc-spaces} must break down for spectra. The precise place where the argument fails is the passage through $\map(T,\map(f,Z))$. In spectra, this expression has no meaning, because $T$ is a spectrum but $\map(f,Z)$ is a space. So the argument of Proposition \ref{prop:monoidal-loc-spaces} relies precisely on the fact that $\M=sSet$ (or $\M=Top$ in the topological case), so that the SM7 axiom for $\M$ is precisely the same as the pushout product axiom.
\end{remark}

Theorem \ref{bigthm} and Theorem \ref{thm:spitzweck} combine to prove that any monoidal Bousfield localization of spectra preserves $A_\infty$ and $E_\infty$-algebras. In particular, $A_\infty$ and $E_\infty$-algebras are preserved by stable Bousfield localizations such as $L_E$ where $E$ is a homology theory. So our results recover Theorems VIII.2.1 and VIII.2.2 of \cite{EKMM}.

\subsection{Equivariant Spectra}

In order to specialize Corollary \ref{cor:loc-pres-alg-for-semi} to the case of equivariant spectra, where $G$ is a compact Lie group, we must first understand the generating cofibrations. For $Top^G$, the (co)domains of maps in $I$ take the form $((G/H) \times S^{n-1})_{+}$ and $((G/H) \times D^{n})_{+}$ for $H$ a closed subgroup of $G$, by Definition 1.1 in \cite{mandell-may-equivariant}. For $G$-spectra, we first need a new piece of notation. For any finite dimensional orthogonal $G$-representation $W$ there is an evaluation functor $Ev_W:\Sp^G \to Top^G$. This functor has a left adjoint $F_W$ (see Proposition 3.1 in \cite{hovey-white} for more details). The (co)domains of maps in $I$ take the form $F_W((G/H)_+ \wedge S^{n-1}_+)$ and $F_W((G/H)_+ \wedge D^{n}_+)$ by Definition 1.11 in \cite{mandell-may-equivariant}, where $W$ runs through some fixed $G$-universe $\cat{U}$. The latter are contractible, and so smashing with them does not make a difference. Observe that the domains of the generating cofibrations are cofibrant. Left Bousfield localization yields the {\em stable model structure}, which we denote $\Sp^G$. That $\Sp^G$ is a monoidal model category with cofibrant objects flat is verified in Proposition III.7.3 of \cite{mandell-may-equivariant}, and may also be deduced from Corollary 4.4 in \cite{hovey-white}. We could also work with the {\em positive stable model structure} $\Sp^G_+$, which has the same weak equivalences as $\Sp^G$, but cofibrations defined by functors $F_W$ where $W^G \neq 0$. The proof that these model structures are left proper and cellular can be found in the appendix of \cite{gutierrez-white-equivariant}.
For $\M=\Sp^G$ or $\Sp^G_+$, our preservation result (Corollary \ref{cor:loc-pres-alg-for-semi} together with Theorem \ref{thm:PPAxiom}) becomes:

\begin{theorem} \label{thm:localization-preservation-families} Let $G$ be a compact Lie group. In $\Sp^G$ (resp. $\Sp^G_+$), a Bousfield localizations $L_\C$ is monoidal if and only if $\C \wedge F_W((G/H)_+ \wedge S^{n-1}_+)$ is a $\C$-local equivalence for all closed subgroups $H$ of $G$, for all $W$ in the universe (resp. all $W$ such that $W^G\neq 0$), and for all $n$ (resp. $n>0$). Furthermore, such localizations preserve $P$-algebra structure for any $\Sigma$-cofibrant $P$, including any equivariant $E_\infty$-operad $P$.
\end{theorem}

Here a $G$-operad $P$ is called {\em equivariant $E_\infty$} if it is $\Sigma$-free, the spaces $P(n)$ are $G$-CW complexes, and $P(n)^H \simeq \ast$ for all closed subgroups $H$ of $G$. These operads are $\Sigma$-cofibrant with respect to the model structure on $G$-operads transferred from the model structure on $G$-collections $\prod_{n\geq 0}\limits (Top^G)^{\Sigma_n} $ where a morphism $f = (f_n)$ is a weak equivalence (resp. fibration) if $f_n^H$ is a weak equivalence (resp. fibration) in $Top$ for every closed subgroup $H$ of $G$ and every $n$. Note that these operads do {\em not} encode genuine equivariant commutativity. To do that, subgroups of $G\times \Sigma_n$ would need to be considered. In particular, there is not a Quillen equivalence between algebras over an $E_\infty$-operad and commutative equivariant ring spectra. The $N_\infty$-operads of \cite{blumberg-hill} were introduced to encode genuine commutativity in a homotopy coherent way (relative to a choice of a collection of families of subgroups of $G\times \Sigma_n$ for $n\geq 0$) and were constructed in \cite{gutierrez-white-equivariant} as cofibrant replacements of the operad $Com$ in various model structures on the category of $G$-operads, corresponding to the choice of a collection of families of subgroups of $G\times \Sigma_n$.

%Note that this is different from the notion of a cellular $E_\infty$-$G$-operad in the language of \cite{lewis-may-steinberger}. This means the $G$-spaces $P(n)$ are all $G\times \Sigma_n$-CW spaces that are universal spaces of principal $(G,\Sigma_n)$-bundles. These correspond to working with all subgroups of $G\times \Sigma_n$, not just with subgroups of the form $H \times 1$.

Ignoring suspensions, Theorem \ref{thm:localization-preservation-families} demonstrates that monoidal Bousfield localizations are precisely the ones for which $L_\C$ respects smashing with $(G/H)_+$ for all subgroups $H$. In this light, Theorem \ref{thm:localization-preservation-families} can be seen as a generalization of Proposition \ref{prop:stable-implies-monoidal}, saying that if $L_\C$ respects stabilization with respect to all the objects  $F_W((G/H)_+ \wedge S^{n-1}_+$ then $L_\C$ is monoidal. We think of these monoidal localizations as the ones that can `see' the information of all subgroups. A natural question is: what if $L_\C$ can only `see' the information of some subgroups $H$? To answer this question, we must consider the following model structures, from in Theorem 6.3 in \cite{mandell-may-equivariant} (on spectra either the stable or positive stable model structure can be used):

\begin{defn}
Let $\sF$ be a family of closed subgroups of $G$, i.e. a non-empty set of subgroups closed under conjugation and taking subgroups. Then the \textit{$\sF$-fixed point model structure} on pointed $G$-spaces is a cofibrantly generated model structure in which a map $f$ is a weak equivalence (resp. fibration) if and only if $f^H$ is a weak equivalence (resp. fibration) in $Top$ for all $H \in \sF$. We will denote this model structure by $Top^\sF$. The generating (trivial) cofibrations are $(G/H \times g)_+$, where $g$ is a generating (trivial) cofibration of topological spaces, and $H\in \sF$.

The corresponding cofibrantly generated model structure on $G$-spectra will be denoted $\Sp^\sF$. Again, weak equivalences (resp. fibrations) are maps $f$ such that $f^H$ is a weak equivalence (resp. fibration) of orthogonal spectra for all $H\in \sF$. The generating (trivial) cofibrations are $F_W((G/H)_+ \wedge g)$ as $H$ runs through $\sF$, $g$ runs through the generating (trivial) cofibrations of spaces, and $W$ runs through some $G$-universe $\cat{U}$.
\end{defn}

With the generating cofibrations in hand, Theorem \ref{thm:PPAxiom} implies that monoidal Bousfield localizations in $\Sp^\sF$ are characterized by the property that $\C \wedge (G/H)_+$ is a $\C$-local equivalence for all $H \in \sF$ (again, ignoring suspensions). One can also define $\sF$-fixed point semi-model structures $Oper^\sF$ on the category of $G$-operads by applying the general machinery of Theorem 12.2.A in \cite{fresse-book}. Indeed, \cite{gutierrez-white-equivariant} demonstrates how to define full model structures on these categories of operads, for even more general families of subgroups.

\begin{defn}
Let $E_\infty^\sF$ be the cofibrant replacement for the operad $Com$ in the $\sF$-fixed point semi-model structure on $G$-operads.
\end{defn}

These operads form a lattice (ordered by family inclusion) interpolating between non-equivariant $E_\infty$ (corresponding to the family $\sF=\{e\}$) and equivariant $E_\infty$ (corresponding to the family $\sF=\{All\}$ and denoted $E_\infty^G$). An $E_\infty^\sF$-algebra $X$ has a multiplicative structure on $res_H(X)$ (compatible with the transfers) for every $H\in \sF$. However, $N_H^G(res_H(X))$ need not have a multiplicative structure. These operads $E_\infty^\sF$ have been generalized and further studied in \cite{gutierrez-white-equivariant}, which includes a comparison between these operads and the $N_\infty$-operads of \cite{blumberg-hill}, results about transferred model structures, and rectification results. For now we will focus on how $E_\infty^\sF$-algebra structure interacts with Bousfield localization. First, observe that both $Oper^\sF$ and $\Sp^\sF$ are $Top^\sF$-model structures (in the sense of Definition 4.2.18 in \cite{hovey-book}) and the cofibrancy of $E_\infty^\sF$ is relative to the $\sF$-model structure. Thus, from a model category theoretic standpoint, $E_\infty^\sF$-algebras are best viewed in $\Sp^\sF$. The following two theorems also have formulations for the positive stable model structure, in analogy with Theorem \ref{thm:localization-preservation-families}, that we leave to the reader.

\begin{theorem}
Let $\M = \Sp^G$ and let $\sF$ be a family of closed subgroups of $G$. Assume $F_W((G/H)_+ \wedge S^{n-1}_+) \wedge \C$ is contained in the $\C$-local equivalences for all $H\in \sF$, for all $n$, and for all $W$ in the universe. Then $L_\C$ preserves $E_\infty^{\sF}$-structure.
\end{theorem}

Localizations of the form above are $\sF$-monoidal but not necessarily $G$-monoidal. For this reason, when $X\in E_\infty^G$-alg, $L_\C(X)$ has $E_\infty^\sF$-algebra structure but may not have $E_\infty^G$-algebra structure. More generally, we have the following result, encoding the fact that if we work in $\Sp^\sK$ rather than $\Sp^G$, then localizations are compatible with both $\sK$ and $\sF$. Because there are now two families involved, the localization will preserve algebraic structure corresponding to the meet of these two families in the lattice of families.

\begin{theorem} \label{thm:preservation-families}
Let $\M$ be the $\sK$-fixed point model structure on $G$-spectra and let $\sK'$ be a subfamily of $\sK$. Assume $F_W((G/H)_+ \wedge S^{n-1}_+) \wedge \C$ is contained in the $\C$-local equivalences for all $H\in \sK'$, for all $n$, and for all $W$ in the universe. Then $L_\C$ takes any $E_\infty^\sF$-algebra to an $E_\infty^{\sF \cap \sK'}$-algebra.
\end{theorem}

\begin{proof}
In order to apply Corollary \ref{cor:loc-pres-alg-for-semi}, first forget to the model structure $\Sp^{\sF \cap \sK'}$ and observe that any $E_\infty^\sF$-algebra is sent to a $E_\infty^{\sF \cap \sK'}$-algebra. The hypothesis on $L_\C$ guarantees that $L_\C$ is a monoidal Bousfield localization with respect to the $\sF\cap \sK'$ model structure, and so $E_\infty^{\sF \cap \sK'}$ is preserved.
\end{proof}

\begin{remark} \label{remark:localization-reduces-in-lattice}
It is easy to produce examples of localizations $L_\C$ that reduce $E_\infty^{\sF}$-algebra structure to  $E_\infty^{\sF'}$-algebra structure for any families $\sF' \subsetneq \sF$ of closed subgroups of $G$, by generalizing the Postnikov section \ref{example:postnikov}. For every closed subgroup $H \in \sF \setminus \sF'$, consider a truncation of the spectrum $F_W((G/H)_+ \wedge S^{n-1}_+)$. Localizing with respect to the wedge of these truncation maps will take $E_\infty^\sF$-algebras to $E_\infty^{\sF'}$-algebras, by Theorem \ref{thm:preservation-families}.
\end{remark}

Together, Theorem \ref{thm:preservation-families} and Remark \ref{remark:localization-reduces-in-lattice} resolve the preservation question for operads in the lattice $E_\infty^\sF$. As expected, preservation of lesser algebraic structure comes down to requiring a less stringent condition on the Bousfield localization. The least stringent condition is for $\sF = \{e\}$ and recovers the notion of a stable localization (i.e. one which is monoidal on the category of spectra after forgetting the $G$-action). However, because none of the operads $E_\infty^\sF$ rectify with respect to the $Com$ operad, we do not have preservation results for commutative equivariant ring spectra. For the remainder of the section, we discuss localizations that preserve $E_\infty^\sF$-algebra structure but fail to preserve commutative structure. We begin with the example that motivated this paper, which the author learned from a talk given by Mike Hill at Oberwolfach (the proceedings can be found in \cite{oberwolfach}). A similar example appeared in \cite{mcclure-E-infty-Tate}. Before presenting this motivating example, we must introduce some new notation. 

We have already seen that, given a $G$-space $X$ and a closed subgroup $H$, one may restrict the $G$ action to $H$ and obtain an $H$-space denoted $res_H(X)$. This association is functorial and lifts to a functor $res_H:\Sp^G \to \Sp^H$. This \textit{restriction functor} has a left adjoint $G_+ \wedge_H (-)$, the \textit{induction functor}. We refer the reader to Section 2.2.4 of \cite{kervaire} for more details. If one shifts focus to commutative monoids $Comm_G$ in $\Sp^G$ (equivalently to genuine $E_\infty$-algebras) then there is again a restriction functor $res_H:Comm_G \to Comm_H$ and it again has a left adjoint functor $N_H^G(-)$ called the \textit{norm}. This functor is discussed in Section 2.3.2 of \cite{kervaire}. We are now prepared to present the example from \cite{oberwolfach}, which will be generalized in Example \ref{example:hill-ober-version} below.

\begin{example} \label{example:hill}
There are localizations that destroy genuine commutative structure but that preserve equivariant $E_\infty$-algebra structure. For this example, let $G$ be a (non-trivial) finite group. Consider the reduced real regular representation $\overline{\rho}$ obtained by taking the quotient of the real regular representation $\rho$ by the trivial representation. We write $\overline{\rho}_G=\rho_G - 1$ where 1 means the trivial representation $\R[e]$. Taking the one-point compactification of this representation yields a representation sphere $S^{\overline{\rho}}$. There is a natural inclusion $a_{\overline{\rho}}:S^0\to S^{\overline{\rho}}$ induced by the inclusion of the trivial representation into $\overline{\rho}$. Consider the spectrum $E=\mathbb{S}[a_{\overline{\rho}}^{-1}]$ obtained from the unit $\mathbb{S}$ (certainly a commutative algebra in $\Sp^G$) by localization with respect to $a_{\overline{\rho}}$. We will show that this spectrum cannot be commutative. 

First, for any proper $H < G$, the restriction $\rho_G |_H$ is $[G:H]\rho_H$, so $\overline{\rho}_G |_H = [G:H] \overline{\rho}_H + ([G:H]1 - 1)$. Let $k=[G:H]-1$. Observe that $res_H S^{\overline{\rho}_G} = (S^{\overline{\rho}_H})^{\#[G:H]} \wedge S^k$. It follows that $res_H(E)$ is contractible, because $k > 0$.
% This means that as an $H$-spectrum it is contractible, because there is enough space in the $S^k$ part to deform it to a point. Note, however, that $E$ itself is not locally trivial, since the only fixed points of $a_{\overline{\rho}}$ are 0 and $\infty$, so the map $a_{\overline{\rho}}$ is not equivariantly trivial.

If $E$ were a commutative equivariant ring spectrum, then the counit of the norm-restriction adjunction would provide a ring homomorphism $N_H^G res_H(E)\to E$. But the domain is contractible for every proper subgroup $H$ because $res_H(E)$ is contractible. This cannot be a ring map unless $E$ to be contractible, and we know $E$ is not contractible, because its $G$-fixed points are not contractible.
%the only fixed points of $a_{\overline{\rho}}$ are 0 and $\infty$.
\end{example}

Justin Noel has pointed out that the localization in Example \ref{example:hill} is smashing, hence monoidal (this is clear from the reformulation in Example \ref{example:hill-ober-version}, but was not clear to the author from the formulation above). It follows that the localization preserves $E_\infty$-algebra structure, by Theorem \ref{thm:localization-preservation-families}, and hence takes commutative monoids to $E_\infty$-algebras. Because \textit{any} such $H$ will lead to a failure of $L(\mathbb{S}) = E$ to be commutative, Example \ref{example:hill} is in some sense maximally bad. We now leverage this observation to generalize Example \ref{example:hill}.

Recall from Definition 5.7 of \cite{gutierrez-white-equivariant} that an $\sF$-$N_\infty$-operad $P$ is a generalization of an $N_\infty$-operad, where $P$-algebras have multiplicative norms for all $H\in \sF$. The formulation of Example \ref{example:hill-ober-version} matches the presentation from \cite{oberwolfach} for the case when $G$ is finite, $\sF$ is the family of all subgroups of $G$, and $\sP$ the family of proper subgroups.

\begin{example} \label{example:hill-ober-version}
Let $G$ be a compact Lie group and $\sF$ a family of closed subgroups of $G$. If $X$ is an algebra over an $\sF-N_\infty$-operad then there is a localization $L$ sending $X$ to an $E_\infty^\sF$-algebra. Consider the cofiber sequence $E\sP_+ \to S^0 \to \tilde{E}\sP$ for any family $\sP$ properly contained in $\sF$. Recall the fixed-point property of the space $E\sP$ (discussed very nicely in Section 7 of \cite{schwede-equivariant-lectures}) and deduce:
\[ 
(E\sP_+)^H \simeq  \begin{cases} 
      \ast_+ = S^0 & \text{ if } H\in \sP\\
      \emptyset_+ = \ast & \text{ if } H \notin \sP 
   \end{cases}
\]

For all $H$, the $H$-fixed points of $S^0$ are $S^0$, so that the cofiber obtained by mapping this space into $S^0$ satisfies the following fixed-point property
\[ 
(\tilde{E}\sP)^H \simeq  \begin{cases} 
      \ast & \text{ if } H\in \sP\\
      S^0 & \text{ if } H \notin \sP 
   \end{cases}
\]

Now apply $\Sigma^\infty_+$ to the map $S^0 \to \tilde{E}\sP$. If $G$ is a finite group, and $\sF$ is the family of all subgroups of $G$, then the resulting map $\mathbb{S} \to E$ is the same localization map considered in Example \ref{example:hill} (see Section 7 of \cite{schwede-equivariant-lectures}). Returning to the general case, note that $E$ is not contractible because $E\sP_+$ is not homotopy equivalent to $S^0$ (since $\sP$ is properly contained in $\sF$), though $res_H(E\sP_+)$ \textit{is} homotopy equivalent to $res_H(S^0)$ for any $H\in \sP$. In this formulation it is clear that the map $\mathbb{S}\to E$ is a nullification that kills all maps out of the induced cells $G_+ \wedge_H (H/K)_+ \cong (G/H)_+$ for all $H\in \sP$. %With the characterization of monoidal Bousfield localizations in $\Sp^\sF$, we can see that in order to produce a localization sending $E_\infty^\sF$-algebras to naive $E_\infty$-algebras one need only apply the localization $S^0 \to \tilde{E}\sF$ rather than the localization $S^0 \to \tilde{E}\sP$ for the full family $\sP$ of proper subgroups of $G$.
\end{example}

This localization is monoidal with respect to the $\sF$-model structure, so $E$ is still an $E_\infty^\sF$-algebra by Theorem \ref{thm:preservation-families}. When $G$ is finite, Example \ref{example:hill-ober-version} makes it clear that the localization is simply killing a homotopy element (namely: the Euler class $a_{\overline{\rho}}$ discussed in Section 2.6.3 of \cite{kervaire}). The presentation in Example \ref{example:hill-ober-version} has several benefits of its own: it generalizes to compact Lie groups $G$, it demonstrates that a smaller localization than Example \ref{example:hill} is needed to destroy $\sF-N_\infty$-algebra structure rather than $N_\infty$-algebra structure, and it shows how localization can reduce one's place in the lattice of $\sF-N_\infty$-algebras without reducing it all the way down to $E_\infty^{\sF}$. To see this, observe that the localization $E$ can still admit some multiplicative norms, for subgroups $H \in \sF \setminus \sP$. Hence, $E$ can still be a $\sK-N_\infty$-algebra if the family $\sK$ only intersects $\sP$ in the trivial subgroup (using the fact that $\mathbb{S}$ is an $N_\infty$-algebra for all choices of families of subgroups). Examples such as that of II.2.3 \cite{lewis-may-steinberger} can be used for this purpose.

Motivated by Example \ref{example:hill-ober-version}, we devote the next two sections to determining when a left Bousfield localization must preserve commutative structure. We will see that the key compatibility condition required is that the maps in $\C$ respect the free commutative monoid functor. In the case of equivariant spectra, this will imply that $\C$ respects the multiplicative norm functors.

\section{Bousfield Localization and Commutative Monoids} 
\label{sec:preservation-of-commutative}

In this section we turn to the interplay between monoidal Bousfield localizations and commutative monoids, i.e. algebras over the (non-cofibrant) operad $Com$. In \cite{white-commutative-monoids}, the following theory is developed as Definition 3.1, Theorem 3.2, and Corollary 3.8.

\begin{defn} \label{defn:comm-monoid-ax}
A monoidal model category $\M$ is said to satisfy the \textit{commutative monoid axiom} if whenever $h$ is a trivial cofibration in $\M$ then $h^{\boxprod n}/\Sigma_n$ is a trivial cofibration in $\M$ for all $n>0$.

If, in addition, the class of cofibrations is closed under the operation $(-)^{\boxprod n}/\Sigma_n$ then $\M$ is said to satisfy the \textit{strong commutative monoid axiom}.
\end{defn}

\begin{theorem} \label{thm:commMonModel}
Let $\M$ be a cofibrantly generated monoidal model category satisfying the commutative monoid axiom and the monoid axiom, and assume that the domains of the generating maps $I$ (resp. $J$) are small relative to $(I\otimes \M)$-cell (resp. $(J\otimes \M$)-cell). Let $R$ be a commutative monoid in $\M$. Then the category $CAlg(R)$ of commutative $R$-algebras is a cofibrantly generated model category in which a map is a weak equivalence or fibration if and only if it is so in $\M$. In particular, when $R=S$ this gives a model structure on commutative monoids in $\M$.
\end{theorem}

\begin{corollary} \label{cor:semi-on-comm-alg}
Let $\M$ be a cofibrantly generated monoidal model category satisfying the commutative monoid axiom, and assume that the domains of the generating maps $I$ (resp. $J$) are small relative to $(I\otimes \M)$-cell (resp. $(J\otimes \M$)-cell). Then for any commutative monoid $R$, the category of commutative $R$-algebras is a cofibrantly generated semi-model category in which a map is a weak equivalence or fibration if and only if it is so in $\M$.
\end{corollary}

While these results only make use of the commutative monoid axiom, in practice we usually desire the strong commutative monoid axiom so that in the category of commutative $R$-algebras cofibrations with cofibrant domains forget to cofibrations in $\M$. This is discussed further in \cite{white-commutative-monoids} and numerous examples of model categories satisfying these axioms are given.

In order to apply the corollary above to verify the hypotheses of Corollary \ref{cor:loc-pres-alg-for-semi} we must give conditions on the maps $\C$ so that if $\M$ satisfies the commutative monoid axiom then so does $L_\C(\M)$. As for the pushout product axiom, our method will be to apply Lemma \ref{lemma:left-Quillen-and-loc}, which is just the universal property of Bousfield localization. However, $(-)^{\boxprod n}/\Sigma_n$ is not a functor on $\M$, but rather on $\Arr(\M)$. The following lemma lets us instead work with the functor $\Sym^n:\M \to \M$ defined by $\Sym^n(X)=X^{\otimes n}/\Sigma_n$. This lemma is proved in Appendix A of \cite{white-commutative-monoids}.

\begin{lemma} \label{lemma:boxprod-equiv-if-sym}
Assume that for every $g\in I$, $g^{\boxprod n}/\Sigma_n$ is a cofibration. Suppose $f$ is a trivial cofibration between cofibrant objects and $f^{\boxprod n}/\Sigma_n$ is a cofibration for all $n$. Then $f^{\boxprod n}/\Sigma_n$ is a trivial cofibration for all $n$ if and only if $\Sym^n(f)$ is a trivial cofibration for all $n$.
\end{lemma}

With this lemma in hand, we are ready to prove the main result of this section, regarding preservation of the commutative monoid axiom by Bousfield localization.

\begin{theorem}\label{thm:loc-preserves-cmon-axiom}
Assume $\M$ is a cofibrantly generated monoidal model category satisfying the strong commutative monoid axiom and with domains of the generating cofibrations cofibrant. Suppose that $L_\C(\M)$ is a monoidal Bousfield localization with generating trivial cofibrations $J_\C$. If $\Sym^n(f)$ is a $\C$-local equivalence for all $n \in \N$ and for all $f \in J_\C$, then $L_\C(\M)$ satisfies the strong commutative monoid axiom. 
\end{theorem}

\begin{remark}
The condition of Theorem \ref{thm:loc-preserves-cmon-axiom},that $\Sym^n(f)$ is a $\C$-local equivalence for all $n \in \N$ and for all $f \in J_\C$, is equivalent to the condition that $\Sym(-)$ preserves $\C$-local equivalences between cofibrant objects. The latter condition implies the former because $\Sym^n(f)$ is a retract of $\Sym(f)$, and the maps in $J_\C$ may be assumed to have cofibrant domains, as shown in Theorem \ref{thm:loc-preserves-cmon-axiom}. That the former implies the latter follows from Theorem 5.6 of \cite{batanin-white-eilenberg}, which shows that the existence of a transferred semi-model structure on commutative monoids in $L_\C(\M)$ implies $\Sym(-)$ preserves $\C$-local equivalences between cofibrant objects. This result, together with Theorem \ref{thm:loc-preserves-cmon-axiom}, implies that both conditions are equivalent to existence of a transferred semi-model structure on commutative monoids in $L_\C(\M)$. Furthermore, the existence of this semi-model structure is equivalent to $L_\C(\M)$ satisfying the strong commutative monoid axiom, since the existence of the semi-model structure implies $\Sym(-)$ preserves $\C$-local equivalences, which implies the strong commutative monoid axiom by Theorem \ref{thm:loc-preserves-cmon-axiom}. We have refrained from stating Theorem \ref{thm:loc-preserves-cmon-axiom} as an `if and only if' to match the discussion in \cite{batanin-white-eilenberg} where the converse was first noticed.
\end{remark}

We turn now to the proof of Theorem \ref{thm:loc-preserves-cmon-axiom}, and to several related results. All of these results are meant to find the easiest possible condition to check on $\C$ so that $L_\C(\M)$ satisfies the commutative monoid axiom. Theorem \ref{thm:loc-preserves-cmon-axiom} reduces the problem from having to check the class of all $\C$-local equivalences to only having to check the set $J_\C$ (which, unfortunately, is often mysterious in practice). It is tempting to try to prove Theorem \ref{thm:loc-preserves-cmon-axiom} using Lemma \ref{lemma:left-Quillen-and-loc}, as we did in Theorem \ref{thm:PPAxiom}, since this would reduce the problem to checking the set $J\cup \C$ (which is much less mysterious than $J_\C$). However, $\Sym^n$ is not a left adjoint. One could attempt to get around this by applying Lemma \ref{lemma:left-Quillen-and-loc} with the functor $\Sym:\M \to \CMon(\M)$, but this would require the existence of a model structure on $\CMon(\M)$ in which the weak equivalences are $\C$-local equivalences. As this is what we're trying to prove by obtaining the commutative monoid axiom on $L_\C(\M)$, this approach is doomed to fail. Instead, we opt for a more technical argument, following the techniques of \cite{white-commutative-monoids}.

\begin{proof}[Proof of Theorem \ref{thm:loc-preserves-cmon-axiom}]
It is proven in Appendix A of \cite{white-commutative-monoids} that if $(-)^{\boxprod n}/\Sigma_n$ takes generating (trivial) cofibrations to (trivial) cofibrations, then it takes all (trivial) cofibrations to (trivial) cofibrations. The generating cofibrations of $L_\C(\M)$ are the same as those in $\M$ and $\M$ satisfies the strong commutative monoid axiom, so the class of cofibrations of $L_\C(\M)$ is closed under the operation $(-)^{\boxprod n}/\Sigma_n$.

Suppose, for every generating trivial cofibration $f:X\to Y$ of $L_\C(\M)$, that $\Sym^n(f)$ is a $\C$-local equivalence. Because the domains of the generating cofibrations in $\M$ are cofibrant, the same is true in $L_\C(\M)$ (see Proposition 4.3 in \cite{hovey-comodules}), so we may assume $f$ has cofibrant domain and codomain. In particular, the proof of Lemma \ref{lemma:boxprod-equiv-if-sym} implies $\Sym^n(f)$ is a cofibration, because $f^{\boxprod k}/\Sigma_k$ is a cofibration for all $k$ and the domain $X$ of $f$ is cofibrant.

By hypothesis, $\Sym^n(f)$ is a trivial cofibration of $L_\C(\M)$ for all $n$. We are therefore in the situation of Lemma \ref{lemma:boxprod-equiv-if-sym} and may conclude that $f^{\boxprod n}/\Sigma_n$ is a trivial cofibration for all $n$. We now apply the result from Appendix A of \cite{white-commutative-monoids} to conclude that all trivial cofibrations of $L_\C(\M)$ are closed under the operation $(-)^{\boxprod n}/\Sigma_n$.
\end{proof}

If we know more about $\M$ in the statement of Theorem \ref{thm:loc-preserves-cmon-axiom}, then we can in fact get a sharper condition regarding the generating trivial cofibrations $J_\C$. One way to better understand the trivial cofibrations in $L_\C(\M)$ is via the theory of framings. Definition 4.2.1 of \cite{hirschhorn} defines the \textit{full class of horns on $\C$} to be the class 
\begin{align*} \Lambda(\C) = \{\tilde{f} \boxprod i_n \;|\; f\in \C, n\geq 0\}
\end{align*} 

where $i_n:\partial \Delta[n] \to \Delta[n]$ and $\tilde{f}:\tilde{A} \to \tilde{B}$ is a Reedy cofibration between cosimplicial resolutions. In the case where $\C$ is a set and $\M$ is cofibrantly generated, Definition 4.2.2 of \cite{hirschhorn} defines an \textit{augmented set of $\C$-horns} to be $\overline{\Lambda(\C)} = \Lambda(\C) \cup J$. Finally, 4.2.5 of \cite{hirschhorn} defines a set $\tilde{\Lambda(\C)}$ to be a set of relative $I$-cell complexes with cofibrant domains obtained from $\overline{\Lambda(\C)}$ via cofibrant replacement. Note that, according to the erratum to \cite{hirschhorn}, we do not know that the domains of maps in $\tilde{\Lambda(\C)}$ are cofibrant, but we do know that they are small relative to $I$.
%% need to cite the erratum??

We now advertise the surprising and powerful Theorem 3.11 in \cite{Barnes12LeftAndRight}. This result states that if $\M$ is proper and stable, if the $\C$-local objects are closed under $\Sigma$ (such $L_\C$ are called \textit{stable}), and if $\C$ consists of cofibrations between cofibrant objects then $J_\C$ is $J \cup \Lambda(\C)$. The last hypothesis is a standing assumption for this paper. The key input to Theorem 3.11 is the observation that for such $\M$, a map is a $\C$-fibration if and only if its fiber is $\C$-fibrant.

\begin{corollary} \label{cor:loc-pres-for-stable-proper}
Suppose $\M$ is a stable, proper, simplicial model category satisfying the strong commutative monoid axiom. Suppose that $L_\C$ is a stable and monoidal Bousfield localization such that for all $n \in \N$ and $f\in \C$, $\Sym^n(f)$ is a $\C$-local equivalence. Then $L_\C(\M)$ satisfies the strong commutative monoid axiom.
\end{corollary}

\begin{proof}
By Theorem \ref{thm:loc-preserves-cmon-axiom} we must only check that $\Sym^n$ takes maps in $J_\C = J \cup \Lambda(\C)$ to $\C$-local equivalences. By the commutative monoid axiom on $\M$, maps in $J$ are taken to weak equivalences, so we must only consider maps in $\Lambda(\C)$.

The reason for the hypothesis that $\M$ is simplicial is Remark 5.2.10 in \cite{hovey-book}, which states that the functor $\tilde{A}^m = A\otimes \Delta[m]$ is a cosimplicial resolution of $A$ (at least, when $A$ is cofibrant). We further observe that the model structure on $L_\C(\M)$ is independent of the choice of cosimplicial resolution. Thus, we may take our map in $\Lambda(\C)$ to be of the form $(f \otimes \Delta[m]) \boxprod i_n$ where $f:A\to B$ is in $\C$.

The map $(f \otimes \Delta[m]) \boxprod i_n$ can be realized as the corner map in the diagram
\begin{align*}
\xymatrix{A \otimes \Delta[m] \otimes \partial \Delta[n]_+ \ar[r] \ar[d] \po & B\otimes \Delta[m] \otimes \partial \Delta[n]_+  \ar[d] \ar@/^1pc/[ddr] & \\
A \otimes \Delta[m] \otimes \Delta[n]_+ \ar[r] \ar@/_1pc/[drr] & \dom((f\otimes \Delta[m]) \boxprod i_n) \ar[dr]^{(f\otimes \Delta[m]) \boxprod i_n} & \\
& & B\otimes \Delta[m] \otimes \Delta[n]_+} 
\end{align*}

If we can prove that $(g\otimes K)^{\boxprod n}/\Sigma_n$ is a $\C$-local trivial cofibration for any $\C$-local trivial cofibration $g$ between cofibrant objects then we can apply the same reasoning from the proof of Proposition \ref{prop:helper-for-monoidal-loc} to deduce that the corner map becomes a $\C$-local trivial cofibration after applying $(-)^{\boxprod n}/\Sigma_n$. This reasoning goes by proving that after applying $(-)^{\boxprod n}/\Sigma_n$ the lower curved map and the top horizontal map are $\C$-local trivial cofibrations, so the bottom horizontal map is as well (because it is a pushout), and hence the corner map is a weak equivalence by the two out of three property. This reasoning works because whenever $f$ is a pushout of $g$ then $f^{\boxprod n}/\Sigma_n$ is a pushout of $g^{\boxprod n}/\Sigma_n$ as shown in Appendix A of \cite{white-commutative-monoids}.

Because $g\otimes K$ is a $\C$-local trivial cofibration between cofibrant objects, we may apply Lemma \ref{lemma:boxprod-equiv-if-sym} to reduce the final step to checking that if $\Sym^n(g)$ is a $\C$-local trivial cofibration for all $n$ then so is $\Sym^n(g\otimes K)$. This is proven as Lemma 27 in \cite{gorchinskiy-symmetrizable}.
\end{proof}

When the hypotheses of stability and properness are dropped one can no longer easily write down the set $J_\C$. However, Theorem 4.1.1 (and its proof, notably 4.3.1) in \cite{hirschhorn} demonstrate that the class of maps $X\to L_\C(X)$ are contained in $\tilde{\Lambda(\C)}$-cell. Given a $\C$-local trivial cofibration $g:X_1\to X_2$ between cofibrant objects, applying fibrant replacement $L_\C$ results in a map $L_\C(g)$, that is a weak equivalence between cofibrant objects. An appeal to Ken Brown's lemma on the functor $\Sym^n$ and to the two out of three property reduces the verification that $(-)^{\boxprod n}/\Sigma_n$ takes $g$ to a $\C$-local equivalence to verifying that $(-)^{\boxprod n}/\Sigma_n$ takes $X_i \to L_\C(X_i)$ to $\C$-local equivalences.

Since such maps are in $\tilde{\Lambda(\C)}$-cell, by Appendix A of \cite{white-commutative-monoids} one must only show that maps in $\tilde{\Lambda(\C)}$ are taken to $\C$-local equivalences by $(-)^{\boxprod n}/\Sigma_n$ (that they are taken to cofibrations is immediate by the strong commutative monoid axiom on $\M$). This observation leads to the following result, which we have recently learned was independently discovered as Theorem 28 in version 3 of the preprint \cite{gorchinskiy-symmetrizable}.

\begin{theorem}\label{thm:loc-preserves-cmon-axiom-simplicial}
Suppose $\M$ is a cofibrantly generated, simplicial model category satisfying the strong commutative monoid axiom and with domains of the generating cofibrations cofibrant. Suppose that for all $n \in \N$ and $f\in \C$, $\Sym^n(f)$ is a $\C$-local equivalence. Then $L_\C(\M)$ satisfies the strong commutative monoid axiom.
\end{theorem}

As the proof of this Theorem appears in \cite{gorchinskiy-symmetrizable}, we will content ourselves with the sketch of the proof given above and we refer the interested reader to \cite{gorchinskiy-symmetrizable} for details. With a careful analysis of $\tilde{\Lambda(\C)}$ the author believes one could remove the need for $\M$ to be simplicial. However, lacking equations of the sort found in Remark 5.2.10 of \cite{hovey-book}, he does not know how to proceed.

\begin{remark}
The commutative monoid axiom has a natural generalization to an arbitrary operad $P$. The proof of Proposition 7.6 in \cite{harper-operads} demonstrates a precise hypothesis on $\M$ so that $P$-algebras inherit a model structure, namely that for all $A\in P$-alg and for all $n$, $P_A[n] \otimes_{\Sigma_n} (-)^{\Box n}$ preserves trivial cofibrations (where $P_A$ is the enveloping operad). If these hypotheses are only satisfied for cofibrant $A$ then $P$-alg inherits a semi-model structure. We hope in the future to study the types of localizations that preserve these axioms, so that Corollary \ref{cor:loc-pres-alg-for-semi} may be applied to deduce preservation results for arbitrary operads $P$. We conjecture that the correct condition on a localization is that for all $f\in \C$, for all $A\in P$-alg, and for all $n$, then $P_A[n]\otimes_{\Sigma_n} f^{\boxprod n}$ is contained in the $\C$-local equivalences. Assuming a $P$-algebra analogue of Lemma \ref{lemma:boxprod-equiv-if-sym}, the proof of Corollary \ref{cor:loc-pres-for-stable-proper} will go through, if we assume $P_A[n]\otimes_{\Sigma_n} f^{\boxprod n}$ is contained in the $\C$-local equivalences, for all $f$ of the form $g\otimes K$ where $g\in \C$ and $K$ is a simplicial set.
\end{remark}

\begin{remark} \label{remark:batanin-white}
Theorem \ref{thm:loc-preserves-cmon-axiom} also has a converse, that the author discovered in joint work with Michael Batanin \cite{batanin-white-eilenberg} (Theorem 5.6 and Example 5.9). For nicely behaved model categories, including all examples considered in this paper, the following are equivalent:
\begin{enumerate}
\item $L_\C$ preserves $P$-algebras,
\item $P$-alg($L_\C(\M))$ admits a transferred semi-model structure from $L_\C(\M)$,
\item $L_\C$ lifts to a localization of $P$-algebras (inverting the maps $P(\C)$),
\item $U$ preserves local equivalences,
\end{enumerate}
and any of these statements implies $P(-)$ preserves $\C$-local equivalences between cofibrant objects. A dual result, for the situation of right Bousfield localization, appears in \cite{white-yau4}. It follows that, for any of the situations from Theorem \ref{thm:loc-preserves-cmon-axiom}, \ref{cor:loc-pres-for-stable-proper}, or \ref{thm:loc-preserves-cmon-axiom-simplicial}, $L_\C$ preserves commutative monoids if and only if $\Sym(-)$ preserves $\C$-local equivalences between cofibrant objects. Note that the condition that the objects be cofibrant is no obstacle, in any model category satisfying the strong commutative monoid axiom, since $\C$ can be taken to be a set of cofibrations between cofibrant objects, and $\Sym^n(X)$ is cofibrant whenever $X$ is cofibrant, if the commutative monoid axiom is satisfied. This follows from the filtration on $\Sym^n(\emptyset)\to \Sym^n(X)$ from Lemma A.3 of \cite{white-commutative-monoids}. Although the positive stable model structure only satisfies the (weak) commutative monoid axiom, one can use the positive flat stable model structure to prove all statements needed for the positive stable model structure, as demonstrated in \cite{white-commutative-monoids}.
\end{remark}
%% Know: If X is cofibrant then Sym^n(X) is cofibrant.
% Proof: Write it as the composite of Sym^n(\emptyset) = Q_0^n \to \dots \to Sym^n(X) = Q^n_n. Each map is formed as a pushout of (powers of X) tensored with pushout-products of $\emptyset \to X$ (which are all cofibrations because cofibrations are symmetrizable). And because X is cofibrant, powers of X are cofibrant in M.

\section{Preservation of Commutative Monoids} 
\label{sec:commutative-preservation-applications}

We turn now to the question of preservation under Bousfield localization of commutative monoids. We will be applying Theorem \ref{thm:loc-preserves-cmon-axiom} and Corollary \ref{cor:loc-pres-alg-for-semi} for this purpose in a moment, but we first remark on a simpler case where the hypotheses of Theorem \ref{thm:loc-preserves-cmon-axiom} are not necessary. 

\subsection{Spectra}

Preservation of commutative monoids by monoidal Bousfield localizations is easy in certain categories of spectra, because of the property that for all cofibrant $X$ in $\M$, the map $(E\Sigma_n)_+ \wedge_{\Sigma_n} X^{\wedge n} \to X^{\wedge n}/\Sigma_n$ is a weak equivalence. This property was first noticed in \cite{EKMM}, and we will now discuss it more generally.

Recall that two operads $O$ and $P$ are said to satisfy \textit{rectification} if $P$-alg and $O$-alg are Quillen equivalent model categories. In \cite{white-commutative-monoids}, we introduced the \textit{rectification axiom}, which states that if $Q_{\Sigma_n}S\to S$ is cofibrant replacement for $S$ in $\M^{\Sigma_n}$ then for all cofibrant $X$ in $\M$, the map $Q_{\Sigma_n}S \otimes_{\Sigma_n} X^{\otimes n} \to X^{\otimes n}/\Sigma_n$ is a weak equivalence (this is the natural generalization of the property from \cite{EKMM} mentioned above, and was further generalized in \cite{white-yau3}). Observe that this property automatically holds on $L_\C(\M)$ if it holds on $\M$, because the cofibrant objects are the same and the weak equivalences are contained in the $\C$-local equivalences. We now prove that in the presence of the rectification axiom, preservation results for commutative monoids are particularly nice.

\begin{theorem} \label{thm:preservation-if-rectification}
Let $QCom$ denote a $\Sigma$-cofibrant replacement of $Com$ in $\M$. Let $\M$ be a monoidal model category satisfying the conditions of Theorem 4.6 of \cite{white-commutative-monoids}, so that the rectification axiom implies that $QCom$ and $Com$ rectify. Let $L_\C$ be a monoidal Bousfield localization. Then $L_\C$ preserves commutative monoids. In particular:

\begin{itemize}
\item For positive (flat) symmetric spectra, positive (flat) orthogonal spectra, or $\mathbb{S}$-modules, $QCom$ is $E_\infty$ and any monoidal Bousfield localization preserves strict commutative ring spectra.
\item For positive (flat) $G$-equivariant orthogonal spectra, $QCom$ is $E_\infty^G$ and any monoidal Bousfield localization preserves strict commutative equivariant ring spectra.
\end{itemize}
\end{theorem}

\begin{proof}
Let $E$ be a commutative monoid, so in particular $E$ is a $QCom$ algebra via the map $QCom \to Com$. Because $QCom$ is $\Sigma$-cofibrant, $QCom$-algebras in both $\M$ and $L_\C(\M)$ inherit semi-model structures, so Corollary \ref{cor:loc-pres-alg-for-semi} implies $L_\C(E)$ is weakly equivalent to some $QCom$-algebra $E_Q$. The rectification  axiom in $L_\C(\M)$ now implies $E_Q$ is weakly equivalent to a commutative monoid $\hat{E}$.
\end{proof}

Currently, this result is only known to apply to the categories of spectra listed in the statement of the theorem. We conjectured in \cite{white-commutative-monoids} that the rectification axiom implies rectification between $QCom$ and $Com$ for general $\M$. If this conjecture is proven then the theorem will apply to all $\M$ satisfying the rectification axiom. Even if the conjecture is false, the following proposition demonstrates that when $\M$ satisfies the rectification axiom then the conditions of Theorem \ref{thm:loc-preserves-cmon-axiom} are satisfied and so any monoidal localization preserves commutative monoids.

\begin{prop} \label{prop:rectification-implies-sym}
Suppose $\sN$ is a monoidal model category satisfying the rectification axiom. Then $\Sym^n(-)$ takes trivial cofibrations between cofibrant objects to weak equivalences.

In particular, if $L_\C(\M)$ is a monoidal Bousfield localization and $\M$ satisfies the rectification axiom, then $L_\C$ preserves commutative monoids.
\end{prop}

\begin{proof}
The first part is proven as Proposition 4.6 in \cite{white-commutative-monoids}, and we refer the reader there for a proof. For the second part, we apply the first part with $\sN = L_\C(\M)$, using our observation that the rectification axiom holds on $L_\C(\M)$ whenever it holds on $\M$. Thus, $\Sym^n:L_\C(\M)\to L_\C(\M)$ takes $\C$-local trivial cofibrations between cofibrant objects to $\C$-local equivalences. In particular, the hypotheses of Theorem \ref{thm:loc-preserves-cmon-axiom} are satisfied and we may deduce from Corollary \ref{cor:loc-pres-alg-for-semi} that $L_\C$ preserves commutative monoids.
\end{proof}

\subsection{Spaces}

We turn our attention now to simplicial sets and topological spaces. Rectification is known to fail (see Example 4.4 in \cite{white-commutative-monoids}), so even though all localizations are monoidal we may not apply the result above. For spaces the path connected commutative monoids are weakly equivalent to generalized Eilenberg-Mac Lane spaces, i.e. products of Eilenberg-Mac Lane spaces. Preservation of commutative monoids has been proven for pointed CW complexes as Theorem 1.4 in \cite{eternal-preprint}.

\begin{theorem} \label{thm:GEMs}
Let $\M$ be the category of pointed CW complexes. Let $\C$ be any set of maps. Then $\Sym(-)$ preserves $\C$-local equivalences and $L_\C$ sends GEMs to GEMs.
\end{theorem}

As a special case of this theorem, we recover classical results of Bousfield, e.g. parts of Theorem 5.1 and Lemma 9.8 from \cite{bous-spaces}. The proof of Theorem \ref{thm:GEMs} is based on work of Dror Farjoun,  Chapter 4 of \cite{farjoun}, in the setting of $\M = sSet$. That work is generalized in \cite{white-commutative-monoids} to hold for the category of $k$-spaces. So we may extend the theorem above to $k$-spaces as well. Observe that the theorem above implies both $sSet$ and $k$-spaces satisfy the conditions of Theorem \ref{thm:loc-preserves-cmon-axiom} because $\Sym^n$ is a retract of $\Sym$. 

\begin{theorem}
Let $\M$ be either simplicial sets or $k$-spaces. Then every Bousfield localization preserves GEMs.
\end{theorem}

Thus, we have extended the result above and Theorem 4.B.4 in \cite{farjoun} to a wider class of topological spaces than spaces having the homotopy type of a CW complex.

\subsection{Chain Complexes}

When $k$ is a field of characteristic zero, there are model structures on $Ch(k)_{\geq 0}$ (resp. $Ch(k)$) where the weak equivalences are quasi-isomorphisms, the fibrations are degreewise surjections, and the cofibrations are the degreewise split monomorphisms. All operads are $\Sigma$-cofibrant, hence all operad-algebras are preserved by any monoidal Bousfield localization. That cofibrant objects are flat is an easy exercise, using the observation that, for every cofibrant $A$, the map $A \otimes QX \to A\otimes X$ induced by cofibrant replacement (i.e. projective resolution) is a quasi-isomorphism.
%This is why taking $B\otimes^L X$ only requires one of them to be cofibrant (comes up in Kunneth theorem).

\begin{proposition}
Let $k$ be a field of characteristic zero. The only Bousfield localizations of $Ch(k)_{\geq 0}$ are truncations.
\end{proposition}

\begin{proof}
Over any PID, the homotopy type is determined by $H_*$, so this means adding weak equivalences is equivalent to killing some object. Thus, all localizations are nullifications. All objects are wedges of spheres, and killing $k^2$ in degree $n$ is the same as killing $k$ in degree $n$. Thus, the localization is completely determined by the lowest dimension in which the first nullification occurs. The localization is therefore equivalent to $0\to V$ where $V$ is the sphere on $k$ in that dimension. 
\end{proof}

\begin{corollary}
All Bousfield localizations of $Ch(k)_{\geq 0}$ are monoidal and hence preserve algebras over any operad $P$.
\end{corollary}

% Need to argue that you don't need $\Sigma$-cofibrancy, because you only care about this $((\emptyset \to Z)\boxprod f^{\boxprod n})_{\Sigma_n}$ and so Lemma 2.5.2 in \cite{bm06} does the job regardless of the choice of weak equivalences.

\begin{remark}
For unbounded chain complexes, truncations need not preserve algebraic structure. For example, if $f:S^{-2}\to D^{-3}$ gets inverted then just as with the Postnikov Section, an algebra will be taken to an object with no unit. 
\end{remark}

Quillen proved in Proposition 2.1 of Appendix B of \cite{quillen-rational-annals} that bounded chain complexes over a field of characteristic zero satisfies the commutative monoid axiom. The proof that all quasi-isomorphisms are closed under $\Sym^n$ goes via cofiber and the 5-lemma on homology groups. The key observation is that $\Sym^n(-)$ preserves group isomorphisms. The same proof demonstrates that $\Sym^n$ preserves $\C$-local equivalences for all $L_\C$ as above. Hence, all Bousfield localizations of $Ch(k)_{\geq 0}$ preserve commutative differential graded algebras. Of course, this can also be seen directly from the description of $L_\C$ as a truncation.

\subsection{Equivariant Spectra}

We conclude this section by returning to our motivating example, Example \ref{example:hill}. Throughout this section, $G$ is a finite group, since otherwise we do not know how to transfer a semi-model structure to commutative equivariant ring spectra. Note that when it was discovered, Example \ref{example:hill} represented a potential gap in the proof of the Kervaire Invariant One Theorem (because the spectrum $\Omega = D^{-1}MU^{(4)}$ needed to be commutative for the computations in \cite{kervaire}). Thankfully, the following theorem from \cite{hill-hopkins} demonstrates that $\Omega$ was indeed commutative.

\begin{theorem} \label{thm:hill-hopkins-preservation} Let $G$ be a finite group. Let $L$ be a localization of equivariant spectra. If for all $L$-acyclics $Z$ and for all subgroups $H$, $N_H^GZ$ is $L$-acyclic, then for all commutative $G$-ring spectra $R$, $L(R)$ is a commutative $G$-ring spectrum.
\end{theorem}

The hypothesis in this theorem is designed so that the proof in \cite{EKMM} regarding preservation of $E_\infty$-structure under localization (i.e. via the skeletal filtration) may go through. We now specialize our preservation result to the context of $G$-spectra, by combining Theorem \ref{bigthm}, Theorem \ref{thm:loc-preserves-cmon-axiom-simplicial}, and Remark \ref{remark:batanin-white}. Recall from Proposition \ref{prop:stable-implies-monoidal} that monoidal localizations are stable. In order to have a transferred semi-model structure on commutative monoids, we need to work with either the positive stable model structure on $G$-spectra (of Theorem 14.2 of \cite{MMSS}), the positive flat stable model structure (of Theorem 2.3.27 of \cite{stolz-thesis}), or the positive complete stable model structure (of Proposition B.4.1 of \cite{kervaire}). The proof that these model structures are left proper and cellular can be found in the appendix of \cite{gutierrez-white-equivariant}.

\begin{theorem} \label{thm:recover-hill-preservation}
Let $\M$ denote any of the positive model structure on $G$-spectra discussed above. Suppose $L_\C$ is a monoidal left Bousfield localization. Then the following are equivalent:
\begin{enumerate}
\item $L_\C$ preserves commutative equivariant ring spectra,
\item $\Sym^n(-)$ preserves local equivalences between cofibrant objects, for all $n$,
\item $\Sym^n(-)$ takes maps in $\C$ to local equivalences, and 
\item $\Sym^n(-)$ preserves $L$-acyclicity for all $n$.
\end{enumerate}
\end{theorem}
% My condition regarding Sym^n implies Hill's condition about R-acyclicity. To see this, consider $\Sym^n(aR \to \ast)$ and you see that $\Sym^n(aR)$ is $R$-acyclic (since $\Sym^n(\ast)\cong B\Sigma_n$ is a wedge of Eilenberg-Maclane spaces and is $R$-acyclic), hence $\Sym^n(-)$ takes all $R$-acyclics to $R$-acyclics.

That preservation of $L$-acyclics is the same as preservation of $L$-local equivalences as can be seen via the rectification axiom and the property that cofibrant objects are flat, but it is easier to observe that this equivalence holds for any stable localization in any stable model category (because consideration of cofibers allows one to reduce to the study of nullifications). In \cite{hill-hopkins}, several equivalent conditions are given in order for a localization to preserve commutative structure. Condition (4) above is equivalent to the condition that, for all $L$-acyclics $Z$ and for all subgroups $H$, $N_H^GZ$ is $L$-acyclic. Hence, Theorem \ref{thm:recover-hill-preservation} sharpens Theorem \ref{thm:hill-hopkins-preservation} to make it an `if and only if' result. 

Another equivalent formulation states that preservation occurs whenever the functors $(E_G \Sigma_n)_+ \wedge_{\Sigma_n} (-)^{\wedge n}$ preserve $L$-acyclicity. This condition can be verified via the skeletal filtration of $E_G \Sigma_n$ into a homotopy colimit of $\Sigma_n$-free $G \times \Sigma_n$ sets $T$ of the form $(G \times \Sigma_n)/\Gamma$ where $\Gamma$ is the graph of a subgroup. This formulation of what is required for $L$ to preserve commutativity is at the heart of the arguments in \cite{kervaire} and \cite{blumberg-hill} and allows for the preservation machinery to be extended to $N_\infty$-operads in \cite{gutierrez-white-equivariant}. The condition is analogous to the non-equivariant condition that functors $(E \Sigma_n)_+ \wedge_{\Sigma_n} (-)^{\wedge n}$ preserve $L$-acyclicity.

%Pf: First note that, when Z is R-acyclic, then R \wedge Z is contractible (here G acts diagonally). Next, consider R wedge (E_G \Sigma_n)_+ \wedge_{\Sigma_n} X^n. Now filter E_G \Sigma_n into a hocolimit of \Sigma_n-free G x Sigma_n sets T (indeed, can arrange it so that T is (G x Sigma_n)/Gamma for some graph Gamma, as we know because cof rep in that model structure). We must show that T_+ \wedge_{\Sigma_n} X^n is R-acyclic. Now, the whole point of norms is that this is equivalent to G_+ \wedge_H N_H^G(res_H(X^n)). You can take K fixed points for any K and this is a point, so itâ€™s equivariantly contractible. 

The Appendix to \cite{blumberg-hill} proves that, for any complete $N_\infty$-operad $P$ whose spaces have the homotopy type of $G\times \Sigma_n$-CW complexes, then $P$-algebras are Quillen equivalent to commutative monoids. Hence, Theorem \ref{thm:preservation-if-rectification} implies the same conditions from Theorem \ref{thm:recover-hill-preservation} are equivalent to preservation of $P$-algebras for any (hence all) complete $N_\infty$-operad $P$ whose spaces have the homotopy type of $G\times \Sigma_n$-CW complexes. Preservation results for non-complete $N_\infty$-operads and for $\sF-N_\infty$-operads can be found in Section 7 of \cite{gutierrez-white-equivariant}.

\section{Bousfield localization and the monoid axiom}
\label{sec:localization-and-monoid-ax}

Recall that the monoid axiom is required to transfer a full model structure to the category of monoids in a monoidal model category \cite{SS00}. However, Theorem \ref{thm:spitzweck} demonstrates that there is a transferred semi-model structure even if the monoid axiom is not satisfied. It follows from Corollary \ref{cor:loc-pres-alg-for-semi}, that our preservation results do not require $L_\C(\M)$ to satisfy the monoid axiom. However, %having found conditions so that the pushout product axiom, commutative monoid axiom, and property that cofibrant objects are flat transfer to $L_\C(\M)$, we feel we should include a word on how to obtain the monoid axiom for $L_\C(\M)$ in case the reader is interested in studying the monoidal model category $L_\C(\M)$ for a purpose other than the preservation of operad-algebra structure. T
the monoid axiom is an important part of the study of monoidal model categories, with many applications beyond the ability to transfer a model structure to monoids, and in this section we provide a result that guarantees it holds on $L_\C(\M)$.

We remark that Proposition 3.8 of \cite{BarnesSplitting} proves that $L_\C(\M)$ inherits the monoid axiom from $\M$ if $L_\C$ takes a special form similar to localization at a homology theory. In contrast, our result will place no hypothesis on the maps in $\C$ at all, beyond our standing hypothesis that these maps are cofibrations. We additionally remark that the preprint \cite{dmitri} has independently considered the question of when Bousfield localization preserves the monoid axiom, towards the goal of rectification results in general categories of spectra. 

In order to understand when Bousfield localization will preserve the monoid axiom we must introduce a definition, taken from \cite{batanin-berger}. Note that this is a different usage of the term $h$-cofibration than the usage in \cite{EKMM} where it means `Hurewicz cofibration.' The meaning here is for `homotopical cofibration' for reasons which will become clear.

\begin{defn}
A map $f:X\to Y$ is called an \textit{$h$-cofibration} if the functor $f_!:X/\M \to Y/\M$ given by cobase change along $f$ preserves weak equivalences. Formally, this means that in any diagram as below, in which both squares are pushout squares and $w$ is weak equivalence, then $w'$ is also a weak equivalence:
\[
\xymatrix{X \ar[r] \ar[d]_f & A \ar[r]^w \ar[d] & B\ar[d]\\
Y \ar[r] & A' \ar[r]_{w'} & B'}
\]
\end{defn}

It is clear that any trivial cofibration is an $h$-cofibration, by the two out of three property. If $\M$ is left proper then any cofibration is an $h$-cofibration (because $A\to A'$ is automatically a cofibration if $f$ is). In fact, the converse holds as well, and is proven in Lemma 1.2 of \cite{batanin-berger}. Lemma 1.3 proves that $h$-cofibrations are closed under composition, pushout, and finite coproduct.

Now let $\M$ be left proper. An equivalent characterization of an $h$-cofibration is as a map $f$ such that every pushout along $f$ is a homotopy pushout (this the version of the definition above was independently discovered in \cite{white-thesis}). Proposition 1.5 in \cite{batanin-berger} proves that $f$ is an $h$-cofibration if and only if there is a factorization of $f$ into a cofibration followed by a \textit{cofiber equivalence} $w:W\to Y$, i.e. for any map $g:W\to K$ the right-hand vertical map in the following pushout diagram is a weak equivalence:
\begin{align*}
\xymatrix{W \ar[r] \ar[d]_w \po & K \ar[d] \\
X \ar[r] & T}
\end{align*}

We will make use of these various properties of $h$-cofibrations in this section. The purpose for introducing $h$-cofibrations is to make the following definition, which should be thought of as saying that the cofibrations in $\M$ behave like inclusions of closed neighborhood deformation retracts of topological spaces.

\begin{defn}
$\M$ is said to be \textit{$h$-monoidal} if for each (trivial) cofibration $f$ and each object $Z$, $f\otimes Z$ is a (trivial) $h$-cofibration.
\end{defn}

We will find conditions so that Bousfield localization preserves $h$-monoidality, and we will then use this to deduce when Bousfield localization preserves the monoid axiom. In \cite{batanin-berger}, $h$-monoidality is verified for the model categories of topological spaces, simplicial sets, chain complexes over a field (with the projective model structure), symmetric spectra (with the stable projective model structure), and several other model categories not considered in this paper. We now verify $h$-monoidality for the remaining model structures of interest in this paper. We remind the reader that an \textit{injective model structure} has weak equivalences and cofibrations defined levelwise, and fibrations defined by the right lifting property.

\begin{proposition} \label{prop:h-monoidal-sym-spec}
The following eight model structures on symmetric spectra are $h$-monoidal (4 stable and 4 unstable model structures):
\begin{enumerate}
\item The levelwise projective (stable) model structure (of Theorem 5.1.2 in \cite{hovey-shipley-smith}, see also Proposition 1.14 of \cite{batanin-berger}).
\item The positive (stable) model structure (of Theorem 14.1 and 14.2 in \cite{MMSS}).
\item The flat (stable) model structure (of Proposition 2.2 and Theorem 2.4 in \cite{shipley-positive}, there called the $S$-model structure).
\item The positive flat (stable) model structure (obtained by redefining the cofibrations from the model structure above to be isomorphisms in level 0, see Proposition 3.1 in \cite{shipley-positive}).
\end{enumerate}
\end{proposition}

\begin{proof}
We appeal to Proposition 1.9 in \cite{batanin-berger}, and make use of the injective (or injective stable for (5)-(8)) model structure on symmetric spectra, introduced in Definition 5.1.1 (resp. after Definition 5.3.6) of \cite{hovey-shipley-smith}. The references above prove that all eight of the model structures above are monoidal and that both injective model structures are left proper (e.g. because all objects are cofibrant). The final condition in Proposition 1.9 is that for any (trivial) cofibration $f$ and any object $X$, the map $f\otimes X$ is a (trivial) cofibration in the corresponding injective model structure. The cofibration part of this is Proposition 4.15(i) in version 3 of Stefan Schwede's book project \cite{schwede-book-symmetric-spectra}, since for all eight of the model structures above the cofibrations are contained in the flat cofibrations and for any $X$ the map $\emptyset \to X$ is an injective (a.k.a. levelwise) cofibration. The trivial cofibration part is Proposition 4.15(iv) in \cite{schwede-book-symmetric-spectra}, which includes statements for both levelwise and stable weak equivalences.
\end{proof}

We turn now to orthogonal and equivariant orthogonal spectra. We first need a lemma regarding the existence of injective model structures. Recall that $\Delta$-generated spaces are a locally presentable category of topological spaces that admits a combinatorial model structure \cite{dugger-delta-generated, Fajstrup-Rosicky}. Let $Sp^O_{\Delta}$ denote orthogonal spectra built on $\Delta$-generated spaces, i.e. where each space in the spectrum is a $\Delta$-generated space. Let $G$ be a compact Lie group and let $GSp^O_{\Delta}$ denote $G$-equivariant orthogonal spectra built on $\Delta$-generated spaces. 

\begin{lemma}
The following model structures exist and are left proper and combinatorial: the levelwise injective model structure on $Sp^O_{\Delta}$, the stable injective model structure on $Sp^O_{\Delta}$, the levelwise injective model structure on $GSp^O_{\Delta}$, and the stable injective model structure on $GSp^O_{\Delta}$.
\end{lemma}

\begin{proof}
Left properness will be inherited from $\Delta$-generated spaces, where it is verified just as for topological spaces \cite{dugger-delta-generated}. For the existence of these model structures, we proceed as in Theorem 5.1.2 and Lemma 5.1.4 of \cite{hovey-shipley-smith}. Verification of lifting and factorization make use of a set $C$ (resp. $tC$) containing a map from each isomorphism class of (trivial) cofibrations $i:X\to Y$ where $Y$ is a countable spectrum. The use of Zorn's Lemma in Lemma 5.1.4 and the requisite countability from Lemmas 5.1.6 and 5.1.7 hold in this setting because of our decision to work with $\Delta$-generated spaces. The rest of Lemma 5.1.4 goes through mutatis mutandis, using properties of topological fibrations and using Lemma 12.2 in \cite{MMSS} when checking that injective cofibrations are closed under smashing with an arbitrary object. 

The sets $C$ and $tC$ serve as generating (trivial) cofibrations. Together with the fact that a category of spectra built on a locally presentable category is again locally presentable, this proves the model structures are combinatorial. 
The stable injective structures are obtained by Bousfield localization in the usual way, which exists because the levelwise structures are left proper and combinatorial.
\end{proof}

\begin{proposition} \label{prop:16-model}
Let $G$ be a compact Lie group and fix a universe $\cat{U}$, that we take to mean a $G$-universe when working equivariantly. Assume all spectra are built on $\Delta$-generated spaces. The following eight model structures (4 stable and 4 unstable) are $h$-monoidal:
\begin{enumerate}
\item The levelwise projective (stable) model structure on $G$-equivariant orthogonal spectra (of Theorem III.2.4 and III.4.2 in \cite{mandell-may-equivariant}).
\item The positive (stable) model structure on $G$-equivariant orthogonal spectra (of Theorem III.2.10 and III.5.3 in \cite{mandell-may-equivariant}).
\item The flat (stable) model structure on $G$-equivariant orthogonal spectra (of Theorem 2.3.13 of in \cite{stolz-thesis}).
\item The positive flat (stable) model structure on $G$-equivariant orthogonal spectra (obtained by redefining the cofibrations from the model structure above to be isomorphisms in level 0, of Theorem 2.3.27 in \cite{stolz-thesis}).
\end{enumerate}
\end{proposition}

Taking $G$ to be the trivial group yields eight model structures on orthogonal spectra, that this proposition proves are $h$-monoidal. They are the levelwise projective (stable) model structure (of Theorem 6.5 and Theorem 9.2 in \cite{MMSS}), the positive (stable) model structure on orthogonal spectra (of Theorem 14.1 and 14.2 in \cite{MMSS}), the flat (stable) model structure on orthogonal spectra (of Proposition 1.3.5 and 2.3.27 in \cite{stolz-thesis}), and the positive flat (stable) model structure on orthogonal spectra (of Proposition 1.3.10 and Theorem 2.3.27 in \cite{stolz-thesis}).

\begin{proof}[Proof of Proposition \ref{prop:16-model}]
The proof proceeds just as it does for Proposition \ref{prop:h-monoidal-sym-spec}, i.e. by comparison to the injective (stable) model structures in each of these settings. For the statement that for any cofibration $f$ and any object $X$, the map $f\otimes X$ is a cofibration in the corresponding injective model structure, we appeal to Lemma 12.2 of \cite{MMSS} (which works equally well in the equivariant context). Finally, we turn to the statement that for any trivial cofibration $f$ and any object $X$, the map $f\otimes X$ is a weak equivalence in the corresponding injective model structure. For the levelwise model structures above this property is inherited from spaces, e.g. by Lemma 12.2 in \cite{MMSS}. For the stable model structures we appeal to the monoid axiom on all of the model structures in the theorem and to the fact that projective (stable) equivalences are the same as injective (stable) equivalences. The monoid axiom has been verified in \cite{stolz-thesis} for all these model structures by Theorems 1.2.54 and 1.2.57 (both originally proven in \cite{MMSS}), 1.3.10, 2.2.46 and 2.2.50 (both originally from \cite{mandell-may-equivariant}), and 2.3.27.
\end{proof}

We return now to the question of the monoid axiom. It is proven in Proposition 2.5 of \cite{batanin-berger} that if $\M$ is left proper, $h$-monoidal, and the weak equivalences in $(\M \otimes I)$-cell are closed under transfinite composition, then $\M$ satisfies the monoid axiom. We will use this to find conditions on $\M$ so that $L_\C(\M)$ satisfies the monoid axiom. First, we improve Proposition 2.5 from \cite{batanin-berger} by replacing the third condition with the hypothesis that the (co)domains of $I$ are finite relative to the class of $h$-cofibrations (in the sense of Section 7.4 of \cite{hovey-book}).

\begin{proposition} \label{prop:helper-monoid-axiom}
Suppose $\M$ is cofibrantly generated, left proper, $h$-monoidal, and the (co)domains of $I$ are finite relative to the class of $h$-cofibrations. Then $\M$ satisfies the monoid axiom.
\end{proposition}

\begin{proof}
We follow the proof of Proposition 2.5 in \cite{batanin-berger}. Consider the class $\{f \otimes Z\;|\; Z\in \M, f\in J\}$. As $\M$ is $h$-monoidal, this is a class of trivial $h$-cofibrations. By Lemma 1.3 in \cite{batanin-berger}, $h$-cofibrations are closed under pushout. By Lemma 1.6 in \cite{batanin-berger}, because $\M$ is left proper, trivial $h$-cofibrations are closed under pushouts (e.g. because weak equivalences are closed under homotopy pushout). In order to prove $\{f \otimes Z\;|\; Z\in \M, f\in J\}$-cell is contained in the weak equivalences of $\M$ we must only prove that transfinite compositions of trivial $h$-cofibrations are weak equivalences.

Consider a $\lambda$-sequence $A_0\to A_1\to \dots \to A_\lambda$ of trivial $h$-cofibrations. Let $j_\beta$ denote the map $A_\beta \to A_{\beta+1}$ in this $\lambda$-sequence. As in Proposition 17.9.4 of \cite{hirschhorn} we may construct a diagram
\begin{align*}
\xymatrix{A_0' \ar[r] \ar[d]_{q_0} & A_1' \ar[r] \ar[d]_{q_1} & \dots \ar[d] \ar[r] & A_\beta' \ar[d]_{q_\beta} \ar[r] & \dots \\ A_0 \ar[r] & A_1 \ar[r] & \dots \ar[r] & A_\beta \ar[r] & \dots}
\end{align*}

in which each $A_\beta'$ is cofibrant, all the maps $A_\beta' \to A_\beta$ are trivial fibrations, and all the maps $A_\beta' \to A_{\beta+1}'$ are trivial cofibrations. Construction of this diagram proceeds by applying the cofibration-trivial fibration factorization iteratively to every composition $j_\beta \circ q_\beta:A_\beta' \to A_\beta \to A_{\beta+1}$ in order to construct $A_{\beta+1}'$. As $j_\beta$ and $q_\beta$ are both weak equivalences, so is their composite and so the cofibration $A_\beta' \to A_{\beta+1}'$ produced by the cofibration-trivial fibration factorization is a weak equivalence by the two out of three property.

We now show that the map $q_\lambda:A_\lambda' \to A_\lambda$ is a weak equivalence, following the approach of Lemma 7.4.1 in \cite{hovey-book}. Consider the lifting problem
\begin{align*}
\xymatrix{X \ar[r] \ar@{^(->}[d]_f & A_\lambda' \ar@{->>}[d]^{q_\lambda} \\
Y \ar[r] & A_\lambda}
\end{align*}

Where $f$ is in the set $I$ of generating cofibrations. Because the domains and codomains of maps in $I$ are finitely presented we know that the map $X\to A_\lambda'$ factors through some finite stage $A_n'$. Similarly, $Y\to A_\lambda$ factors through some finite stage $A_m$. Let $k = \max(n,m)$. The map $A_k'\to A_k$ is a trivial fibration so there is a lift $g:Y\to A_k'$. Define $h:Y\to A_\lambda'$ as the composite with $A_k'\to A_\lambda'$. 
\begin{align*}
\xymatrix{X \ar[r] \ar@{^(->}[d]_f & A_k' \ar[r] \ar[d] & A_\lambda' \ar[d]^{q_\lambda} \\
Y \ar[r] \ar[ur]^g \ar[urr]_(.7)h & A_k \ar[r] & A_\lambda}
\end{align*}

Both triangles in the left-hand square commute by definition of lift. The triangle featuring $g$ and $h$ commutes because it is a composition. So the triangle featuring $f$ and $h$ commutes. The right-hand square commutes by construction of $A_\lambda'$ and $A_\lambda$, so the trapezoid containing $g$ and $q_\lambda$ commutes. Thus, the triangle featuring $h$ and $q_\lambda$ commutes.

The existence of this lift $h$ for all $f\in I$ proves that $A_\lambda' \to A_\lambda$ is a trivial fibration. Now consider that transfinite compositions of trivial cofibrations are always trivial cofibrations, so $A_0'\to A_\lambda'$ is a weak equivalence. Furthermore, the vertical maps $q_0:A_0'\to A_0$ and $q_{\lambda}: A_\lambda'\to A_\lambda$ are trivial fibrations. So by the two out of three property, the map $A_0\to A_\lambda$ is a weak equivalence as required.
\end{proof}

It is shown in \cite{batanin-berger} that the compactness hypothesis of the proposition is satisfied for topological spaces, simplicial sets, equivariant and motivic spaces, and chain complexes. Similarly, it holds for all our categories of structured spectra because the sphere spectrum is $\aleph_0$-compact as a spectrum. Lastly, it holds for all the stable analogues of these structures because the compactness hypothesis is automatically preserved by any Bousfield localization (the set of generating cofibrations of $L_\C(\M)$ is simply $I$ again).

\begin{remark}
The proof above only uses the fact that the maps $j_\beta$ were $h$-cofibrations in order to factor $Y\to A_\lambda$ through some finite stage. So if the (co)domains of $I$ are finite relative to the class of weak equivalences then the proof above demonstrates that weak equivalences are preserved under transfinite composition. This property was already known classically for $sSet$ and $Ch(k)$, but fails for $L_{BP}(\Sp)$, the localization of symmetric spectra with respect to the spectrum $BP$, as discussed in Chapter 7 of \cite{hovey-book}.
\end{remark}

\begin{prop} \label{prop:implies-h-monoidal}
Suppose $\M$ a cofibrantly generated, left proper, $h$-monoidal, and that the (co)domains of $I$ are cofibrant and are finite relative to the class of $h$-cofibrations. Suppose cofibrant objects are flat. Let $L_\C$ be a monoidal Bousfield localization. Then $L_\C(\M)$ is $h$-monoidal.
\end{prop}

\begin{proof}
Suppose $f:A\to B$ is a cofibration in $L_\C(\M)$ and $Z$ is any object of $L_\C(\M)$. We must show $f\otimes Z$ is an $h$-cofibration in $L_\C(\M)$. Because $L_\C(\M)$ is left proper, Proposition 1.5 in \cite{batanin-berger} reduces us to proving that there is a factorization of $f\otimes Z$ into a cofibration followed by a cofiber equivalence $w:X\to B\otimes Z$, i.e. for any map $g:X\to K$ the right-hand vertical map in the following pushout diagram is a $\C$-local equivalence:
\begin{align*}
\xymatrix{X \ar[r] \ar[d]_w \po & K \ar[d] \\
B\otimes Z \ar[r] & T}
\end{align*}

Because $f$ is a cofibration in $\M$, the $h$-monoidality of $\M$ guarantees us that $f\otimes Z$ is an $h$-cofibration in $\M$. Apply the cofibration-trivial fibration factorization in $\M$. Note that this is also a cofibration-trivial fibration factorization of $f\otimes Z$ in $L_\C(\M)$ because cofibrations and trivial fibrations in the two model categories agree. The resulting $w:X\to B\otimes Z$ is a trivial fibration in either model structure. Because $\M$ is left proper we know that the map $w$ is a cofiber equivalence in $\M$ by Proposition 1.5 in \cite{batanin-berger} applied to the $h$-cofibration $f\otimes Z$. So in any pushout diagram as above the map $K\to T$ is a weak equivalence in $\M$, hence in $L_\C(\M)$. Thus, $w$ is a cofiber equivalence in $L_\C(\M)$ and its existence proves $f\otimes Z$ is an $h$-cofibration in $L_\C(\M)$.

Now suppose $f$ were a trivial cofibration in $L_\C(\M)$ to start. We must show that $f\otimes Z$ is a $\C$-local equivalence. We do this first in the case where $f$ is a generating trivial cofibration. By hypothesis, $A$ and $B$ are cofibrant. Apply cofibrant replacement to $Z$:
\begin{align*}
\xymatrix{A\otimes QZ \ar[r] \ar[d] & B\otimes QZ\ar[d]\\
A \otimes Z \ar[r] & B\otimes Z}
\end{align*}

The fact that cofibrant objects are flat in $L_\C(\M)$ implies the vertical maps are $\C$-local equivalences (because $A$ and $B$ are cofibrant) and that the top horizontal map is a $\C$-local equivalence (because $QZ$ is cofibrant). By the two out of three property the bottom horizontal map is a $\C$-local equivalence.

By Lemma 1.3 in \cite{batanin-berger}, the class of $h$-cofibrations is closed under cobase change and retracts. By Lemma 1.6, the class of trivial $h$-cofibrations is closed under cobase change (because $L_\C(\M)$ is left proper). Weak equivalences are always closed under retract. Finally, by Proposition \ref{prop:helper-monoid-axiom} the class of trivial $h$-cofibrations is closed under transfinite composition by our compactness hypothesis on $\M$ (equivalently, on $L_\C(\M)$). So for a general $f$ in the trivial cofibrations of $L_\C(\M)$, realize $f$ as a retract of $g \in J_\C$-cell, so that $g\otimes Z$ is a transfinite composite of pushouts of maps of the form $j\otimes Z$ for $j\in J_\C$. We have just proven that all $j\otimes Z$ are trivial $h$-cofibrations and closure properties imply $g\otimes Z$ and hence $f\otimes Z$ are trivial $h$-cofibrations as well.

\end{proof}

\begin{theorem} \label{thm:monoid-axiom-loc}
Suppose $\M$ is a cofibrantly generated, left proper, $h$-monoidal model category such that the (co)domains of $I$ are cofibrant and are finite relative to the $h$-cofibrations and cofibrant objects are flat. Then for any monoidal Bousfield localization $L_\C$, the model category $L_\C(\M)$ satisfies the monoid axiom.
\end{theorem}

\begin{proof}
Apply Proposition \ref{prop:helper-monoid-axiom} to the category $L_\C(\M)$. By Proposition \ref{prop:implies-h-monoidal}, $L_\C(\M)$ is $h$-monoidal. It is left proper because $\M$ is left proper. To verify the monoid axiom, consider a $\lambda$-sequence of maps that are pushouts of maps in $\{f\otimes Z\;|\; f$ is a trivial cofibration in $L_\C(\M)\}$. Such maps are $h$-cofibrations in $\M$ because $\M$ is $h$-monoidal, $f$ is a cofibration in $\M$, and $h$-cofibrations are closed under pushout. Thus, the hypothesis that the (co)domains of $I$ are finite relative to the $h$-cofibrations in $\M$ is sufficient to construct the lift in Proposition \ref{prop:helper-monoid-axiom} and to prove the transfinite composition part of the proof of the monoid axiom.
\end{proof}

\end{document}